\documentclass{article}[12pt]
\usepackage{amsmath}
\usepackage[pdftex]{graphicx}
\usepackage{textcomp}
\usepackage{pgf,tikz}
\usetikzlibrary{arrows}
\usepackage{amsthm}
\usepackage{enumerate}
\usepackage{mathrsfs}
\usepackage{bbm}
\usepackage{color}
\usepackage{amssymb}

\newtheorem{proposition}{Proposition}
\newtheorem{theorem}{Theorem}
\newtheorem{lemma}{Lemma}
\newtheorem{remark}{Remark}
\newcommand{\p}{\mathbbm{P}}
\newcommand{\e}{\mathbbm{E}}
\newcommand{\N}{\mathbbm{N}}
\newcommand{\R}{\mathbbm{R}}
\newcommand{\A}{\mathcal{A}}

\begin{document}

\title{Partial geodesics on symmetric groups endowed with breakpoint distance\thanks{Partially supported by CNPq, FAPERJ and NSERC. DS holds the Canada Research Chair in Mathematical Genomics.}}

\author{Poly H. da Silva, Arash Jamshidpey and David Sankoff}
 
\maketitle

\begin{abstract}
The notion of partial geodesic (or geodesic patch) was introduced by Jamshidpey et al. in ``Sets of medians in the non-geodesic pseudometric space of unsigned genomes with breakpoints''~\cite{jam14}. In this paper, we study the density of points on  non-trivial partial geodesics between two permutations $\xi_1^{(n)}$ and $\xi_2^{(n)}$ chosen uniformly and independently at random from the symmetric group $S_n$, where $S_n$ is endowed with the breakpoint distance. 

For a permutation $\pi := \pi_1 \  ... \ \pi_n$, any unordered pair $\{\pi_i , \pi_{i+1}\}$, for $i=1, ..., n-1$, is called an adjacency of $\pi$. The set of all adjacencies of $\pi$ is denoted by $\mathcal{A}_\pi$. Denote by $id^{(n)}$ the identity permutation, and let $I_n$ be an arbitrary subset of $\mathcal A_{id^{(n)}}$. We classify the set of all adjacencies of a permutation $\pi \in S_n$ into four types, with respect to $I_n$. Then for a permutation $\xi^{(n)}$ chosen uniformly at random from $S_n$, 
we derive a convergence theorem for the normalized number (after dividing by $n$) of adjacencies of each type in $\xi^{(n)}$
 with respect to $I_n$ (for some random or deterministic choices of $I_n$), as $n\rightarrow \infty$. We also see an application of this convergence theorem to find the appropriate choices of $I_n$. 
 
A geodesic point of $u$ and $v$ in a pseudometric space $(S,\rho)$ is a point $w$ of the space that $\rho(u,w)+\rho(w,v)=\rho(u,v)$. In other words, a point is a geodesic point of $u$ and $v$ if and only if it is located on a partial geodesic between $u$ and $v$. We find an upper bound for the number of permutations $x\in S_n$ for which there exists at least one non-trivial geodesic point between $id^{(n)}$ and $x$, far from both. This partially verifies the conjecture of Haghighi and Sankoff stated in ``Medians seek the corners, and other conjectures''~\cite{haghighi12}, namely we prove that, with high probability, there is no breakpoint median of two permutations $\xi_1^{(n)}$ and $\xi_2^{(n)}$ chosen uniformly and independently at random from $S_n$, far from both of them.
\end{abstract}


\section{Introduction}\label{int}
When there is no duplication, linear unichromosomal genomes are represented by permutations, where each number represents a gene or a marker. To compare two linear unichromosomal genomes with an identical set of genes, one can count the number of their dissimilarities or \textit{breakpoints}. More precisely, for two linear unichromosomal genomes $G$ and $G'$ with the same set of genes, a pair of adjacent genes in $G$ is called a breakpoint of $G$ with respect to $G'$, if these genes are not adjacent in $G'$. It is clear that the number of breakpoints of $G$ with respect to $G'$ is equal to the number of breakpoints of $G'$ with respect to $G$. Introduced by Sankoff and Blanchette in~\cite{sankoff97}, 1997, the breakpoint distance is the number of breakpoints in the set of gene adjacencies of two unichromosomal genomes with an identical set of genes. 

On the other hand, we can use the definition of \emph{median} to compare more than two genomes, namely, having a set of genomes $A=\{G_1,...,G_k\}$ (all genomes are in the symmetric group $S_n$) and a genomic distance $d$, a median of $A$ is a genome that minimizes the total distance function $d_T(\cdot,A):=\sum_{i=1}^k d(\cdot,G_i)$. The minimum value of $d_T(\cdot,A)$ is called the median value of the set $A$. Motivated by the Steiner points, the median problem is the problem of finding a median for a given set $A$ of genomes. The median problem, has been used for the first time  by Sankoff \emph{et al.}~\cite{sankoff96}, for evolutionary models of gene orders. The goal was to obtain more information about the ancestors of a given set of genomes and also to apply it to \emph{small phylogeny problems}. In the small phylogeny problem the topology of the ancestral tree is given and the ancestral nodes (vertices of degree greater than 1) should be estimated such that the total sum of distances over all pairs of neighbours in the tree attains its minimum. 
The tree obtained in this way is the closest tree to preserve the parsimony principle on its paths. 

The median problem has been extensively studied for different genome distances, and for many of them including the breakpoint distance on linear unichromosomal genomes, it is shown that the median problem is NP-hard~\cite{bryant98,caprara03,tannier09,fertincombinatorics}. This paper concerns the breakpoint median problem for linear unichromosomal genomes represented by unsigned permutations. 


Despite of its importance in parsimony-based phylogenetics, the median suffers from several disadvantages. The first one is that it is very hard to find a median for most genomic distances. In fact, as we mentioned, the median problem is NP-hard in many cases. Another problem is that although a median genome may carry valuable information from all given genomes (inputs), it is not necessarily close to the ancestral genome. In other words, it is not a good estimator for the true ancestor. Zheng and Sankoff~\cite{zheng11} provided some simulation studies, for a random model of evolution, showing that their heuristic median does not approximate the ancestor for the long-time evolution of genomes, while for genomes involved in evolution for a shorter period of time, medians may approximate the true ancestor. Later, Jamshidpey and Sankoff~\cite{jam13} proved that when the evolution is modelled by some continuous time random walks on $S_n$ (group of permutations of length $n$), including reversal, DCJ, and transposition random walks (here by transposition we mean the mathematical transposition), until time $cn$ of the evolution, for $c<1/4$, the true ancestor can be approximated asymptotically almost surely by a median while for $c>0.61$, the medians are not close to the true ancestor. They conjectured that the median solutions lose their credibility to approximate the ancestor right after $n/4$. It is worth mentioning that, although the medians will not be useful to approximate the true ancestor for some random evolutionary models, they may still carry some important information about ancestors. More recently, Jamshidpey and Sankoff found all possible positions of asymptotic medians of $k$ random permutations sampled from high speed random walks \cite{jam16, jam17}. Determining all possible locations of medians with respect to a random sample of genomes, their results significantly reduce the median search space for a number of edit distances on groups of permutations or signed permutations. Another obstacle about the median is that they are not unique and different medians may be of considerable distance from each other~\cite{haghighi12}. Then, for a set of genomes having many medians it is not clear which of them is the closest to the ancestor. Still, another concern is that not all the medians carry useful information about the ancestor or input genomes. 
Following some simulation studies, Haghighi and Sankoff~\cite{haghighi12} conjectured that a major proportion of breakpoint medians of $k$ random permutations lie around these $k$ random permutations (corners), and so most of breakpoint medians for random genomes just have information about one of them. However, in their simulations they observed that even it is a minority, there still exist medians that are far from any of these $k$ random permutations, and from the biological point of view, studying these medians is more interesting since they have information from all of the given permutations. They observed that as the size of permutations increases, the proportion of these medians far from the corners decreases. Jamshidpey et al.~\cite{jam14} investigated this conjecture further and found a family of breakpoint median points using the new concept of \emph{accessible points.} This concept may also help us to find a median far from corners. They partially proved the conjecture stated in~\cite{haghighi12}, that the median value of $k$ permutations chosen uniformly at random from $S_n$ is almost $(k-1)(n-1)$ ($2n$ for three random permutations), with high probability, after a convenient rescaling of the breakpoint distance. They showed that any accessible point from a set of $k$ random permutations is an asymptotic median of those $k$ random permutations, with high probability. They proved that any median of $k$ random permutations must take almost all of its adjacencies from at least one of the $k$ random permutations. Making use of this mathematical property in~\cite{jam14}, Larlee \emph{et al.}~\cite{larlee14} proposed a construction for a genome which includes gene order information from all three given genomes such that the total distance is approximately $2.25 n$, where $n$ is the size of the permutations, that is $0.25n$ bigger than the median value.


Motivated by the conjecture of Haghighi and Sankoff in \cite{haghighi12}, one of the objectives of this paper is to study this conjecture starting with two random permutations, as a first step, and in doing so, construct tools and results that can be used later to help in the general problem (for more than two permutations). In particular, in this paper, we study the accessible points of two random permutations. 
We introduce different notions to study the breakpoint median of two or more number of permutations. We provide an equivalence definition for the concept of accessibility of two permutations. Given a subset $I$ of adjacencies of 
the identity permutation $id=id^{(n)}$ (later we call this kind of subsets, segment sets), we classify the set of all adjacencies of the symmetric group $S_n$, with respect to $I$, into four types. Then for a permutation $\xi^{(n)}$ chosen uniformly at random from $S_n$ we compute the expectation and variance of the number of adjacencies of each type in $\xi^{(n)}$. We derive a convergence theorem for the normalized number (after dividing by $n$) of different types of adjacencies of $\xi^{(n)}$
with respect to $I$ (for both random or deterministic choice of $I$). This leads us to discuss further about the possible segment sets $I$ chosen from identity for which one can construct a permutation $\pi$ in the set of all permutations lying on partial geodesics connecting $id$ and $\xi^{(n)}$, denoted by $ \overline{[id, \xi^{(n)}]}$, such that the set of adjacencies of $\pi$ contains $I$ and the remaining adjacencies of $\pi$ are contained in the set of adjacencies of $\xi^{(n)}$. Taking convenient segment sets $I$ (whose size is neither very small nor very big) we can say that 
$\pi$ is located far from $id$ and $\xi^{(n)}$. In this way, we can estimate an upper bound for the probability of existence of a permutation $\pi$ in $\overline{[id,\xi^{(n)}]}$, far from corners. We see that this probability converges to $0$, as $n$ tends to $\infty$.

\section{Preliminaries}\label{prem}
A permutation of length $n$ is a bijection on $[n]:=\{1,...,n\}$. A permutation $\pi$ is denoted by
\[
{1 \ \ 2 \ \ ... \ \ \ n \choose \pi_1 \ \pi_2 \ \ ... \ \ \pi_n},
\]
or simply by $\pi_{1} \ \pi_2 \ \ ... \ \ \pi_{n}$. We represent a linear unichromosomal genome with $n$ genes or markers by a permutation of length $n$. Each number represents a gene or a marker in the genome. The set of all permutations of length $n$ with the function composition operator is a group called the \textit{symmetric group} of order $n$ denoted by $S_n$. We denote by $id:=id^{(n)}$ the identity permutation $1 \ 2 \ 3 \ ... \ n$. For a permutation $\pi := \pi_1 \  ... \ \pi_n$, any unordered pair $\{\pi_i , \pi_{i+1}\}=\{\pi_{i+1} , \pi_i\}$, for $i=1, ..., n-1$, is called an adjacency of $\pi$. We denote by $\mathcal{A}_{\pi}$ the set of all adjacencies of $\pi$ and by $\mathcal A_{x_1,...,x_k}$ the set of all common adjacencies of $x_1,...,x_k\in S_n$. For any $x,y \in S_n$, the \textit{breakpoint distance} (bp distance) between $x$ and $y$ is defined by $d^{(n)}(x,y):=n-1-|\mathcal A_{x,y}|$ which is a pseudometric. 
We say a pseudometric (or a metric) $\rho$ is left-invariant on a group $G$ if for any $x,y,z \in G$, $\rho(x,y)=\rho(zx, zy)$. The bp distance is a left-invariant pseudometric on $S_n$. We say two permutations $\pi$ and $\pi'$ in $S_n$ are equivalent, denoted by $\pi\sim \pi'$, if $d^{(n)}(\pi,\pi')=0$. In other words they are equivalent if $\pi_i=\pi'_{n+1-i}$, for $i=1,...,n$. The equivalence class containing permutation $\pi$ is denoted by $[\pi]$. The set of all equivalence classes of $S_n$ under $\sim$, denoted by $\hat{S}_n:=S_n/\sim$, endowed with $d^{(n)}$ is a metric space.

A discrete metric space $(S, \rho)$ (i.e. a metric space $S$ with metric $\rho: S\times S \rightarrow \N_0:=\N\cup \{0\}$) is said to be a \emph{discrete geodesic space}, if for any two points $x,y\in S$, there exists a finite subset of $S$ containing $x$ and $y$ that is isometric with the discrete line segment $[0, 1, ..., \rho(x, y)]$ ($\N_0$ is endowed with the standard metric $dist(m, n) :=|m-n|$). In other words, it is a geodesic space if for any two points $x,y \in S$ with $\rho(x,y)=k \in \N_0$, there exists a finite chain of length $k$ in $S$, namely $z_0=x,z_1...,z_k=y$, such that $\rho(z_i,z_{i+1})=1$, for $i=0,...,k-1$. Any chain in $S$ with this property is called a geodesic between $x$ and $y$. Indeed, a countable metric space is a discrete geodesic if and only if it is isometric with a connected graph. Of course, one side of this is more obvious. For the other side (sufficiency), construct a graph $G$ from a countable discrete geodesic metric space $(S,\rho)$ whose vertices are points of $S$ and a pair of points $x,y\in S$ are connected by an edge if $\rho(x,y)=1$. The graph $G$ endowed with the graph distance is isometric with $(S,\rho)$, as the shortest paths between two vertices $x,y$ coincide with the geodesics between $x,y\in S$.

When a discrete metric space $(S, \rho)$ is not geodesic, as for the case of $\hat S_n$ endowed  with bp-distance~\cite{jam14}, the concept of a geodesic between two points $x$ and $y$ can be extended to the concept of a \emph{partial geodesic} or \emph{geodesic patch} (p-geodesic)~\cite{jam14}, that is a maximal subset of S containing $x$ and $y$ which is isometric to a subsegment (not necessarily contiguous) of the line segment $[0,...,\rho(x,y)]$. In other words, a p-geodesic between $x$ and $y$ is a maximal chain $z_0=x,z_1,...,z_k=y$ in $S$ such that 
\[
\sum\limits_i \rho(z_i,z_{i+1})=\rho(x,y).
\]
Note that the former form of the definition is very general and can be extended to general metric spaces, i.e. for a general metric space a p-geodesic between two points $x$ and $y$ is the maximal subset of the metric space which can be isometrically embedded into the real interval $[0,\rho(x,y)]$ (where $\R_+$ is endowed with the Euclidean topology). Since, our spaces of interest are the finite symmetric groups, we only work on discrete metric spaces in this paper, and so the second form of the definition for p-geodesics is suitable for us. 

For any two points $x,y$ in an arbitrary metric space $(S, \rho)$ there exists at least one p-geodesic between them, since the trivial chain of length one, $z_0=x,z_1=y$, always exists. If this chain is maximal then the p-geodesic $z_0=x,z_1=y$ is called trivial. Only non-trivial p-geodesics, those containing at least three points of the space, are interesting for us. Any point on a p-geodesic between $x$ and $y$ is called a geodesic point of $x$ and $y$. In the case of permutations (or permutation classes), we also call a geodesic point, a geodesic permutation (or a geodesic permutation class). Note that any geodesic is a p-geodesic, and for any geodesic point of $x$ and $y$, say $z$, we have $\rho(x,y)=\rho(x,z)+\rho(z,y)$.
We denote by $\overline {[x,y]}_S$ the set of all geodesic points of $x$ and $y$ in a metric space $(S,\rho)$, and in particular for $x,y\in S_n$, we denote by $\overline {[x,y]}^*$ or $\overline {[[x],[y]]}^*=\overline {[u,v]}^*$ the set of all geodesic points of $u=[x],v=[y]\in \hat S_n$, that is the set of all permutation classes lying on partial geodesics connecting $[x]$ and $[y]$ in $\hat S_n$. In addition, for $x,y\in S_n$, we denote
\[
\overline {[x,y]}:=\{z\in S_n: d^{(n)}(x,y)=d^{(n)}(x,z)+d^{(n)}(z,y)\}.
\]
In other words, $z\in \overline {[x,y]}$ if and only if $[z]\in \overline {[x,y]}^*$.
For a metric (or pseudometric) space $(S,\rho)$, let us define the total distance of a point $x\in S$ to a finite subset $A\subset S$ by
\[
\rho_T(x,A):=\sum\limits_{y\in A} \rho(x,y).
\]
A \emph{median} of a finite subset $A\subseteq S$  is a point of $S$ (not necessarily unique) whose total distance to $A$ takes the infimum (respectively, minimum for a finite space $S$), i.e.  a point $x\in S$ such that
\[
\rho_T(x,A)=\inf_{y\in S} \rho_T(y,A).
\]
For the finite space $S$, ``\textit{inf}'' is replaced by ``\textit{min}'' in the above definition, that is $x\in S$ is a median of $A$ if it minimizes the total distance function $\rho_T(.,A)$. Furthermore, the median value of $A$, denoted by $\mu(A)$, is the infimum (respectively, minimum) value of the total distance function to $A$. We denote by $\mathcal{M}_{S,\rho}(A)$ the set of all medians of $A\subset S$. In particular, we denote by $d_T^{(n)}(x,A)$ the total breakpoint distance of permutation $x\in S_n$ to $A\subset S_n$, and by $\mathcal{M}_n(A)$ the set of all breakpoint medians of $A$, that is $\mathcal{M}_n(A):=\mathcal{M}_{S_n,d^{(n)}}(A) $. There always exists a median (not necessarily unique) for any subset of a finite metric space, while this is not true for general infinite metric spaces. In the simple case of two points $x$ and $y$ in a general metric space, it is clear from the definition that every median of $x$ and $y$  is a geodesic point of them and vice versa. That is, $\overline {[x,y]}_S$ is the set of medians of $x$ and $y$. 

Medians play an important role in small and large phylogeny problems. In some evolutionary models, at least one of the medians of some species carries valuable information about their first common ancestor or even about the phylogenetic tree. According to some simulation studies, when the symmetric group is endowed with the bp-distance, Haghighi et al.~\cite{haghighi12} conjectured that a major proportion of bp medians of $k$ random permutations lie around these $k$ random permutations (corners). Therefore, it seems hard to find a median far from any of these $k$ random permutations. Jamshidpey et al.~\cite{jam14} investigated this further and found a family of bp median points using the new concept of accessible points. This concept may also help to find a median far from all random permutations. More precisely, let $X$ be a subset of $\hat S_n$. Following \cite{jam14}, we say $z \in \hat S_n$ is $1$-accessible from $X$ if there exists a natural number $m$, a finite sequence $y_1,..., y_m \in X$, and a finite sequence $z_1=y_1, ..., z_m=z \in \hat S_n$ such that $z_{i+1} \in \overline{[z_i,y_{i+1}]}^*$, for $i=1 ... m-1$ (See Fig.~\ref{fig2}). The set of all $1$-accessible points of $X$ is denoted by $Z(X)$. Let $Z_0(X):=X$, and, by induction, for $r \in \N_0$, let $Z_{r+1}(X):=Z(Z_r(X))$. By definition, we have
\[
\bigcup\limits_{x,y \in Z_r(X)}\overline{[x,y]}^* \subset Z_{r+1}(X).
\] 
A permutation class $z$ is accessible from $X$ if there exists a natural number $r$ such that $z \in Z_r(X)$. We denote the set of all accessible points by $\bar Z(X) = \cup_{r \in \N_0} Z_r(X)$.

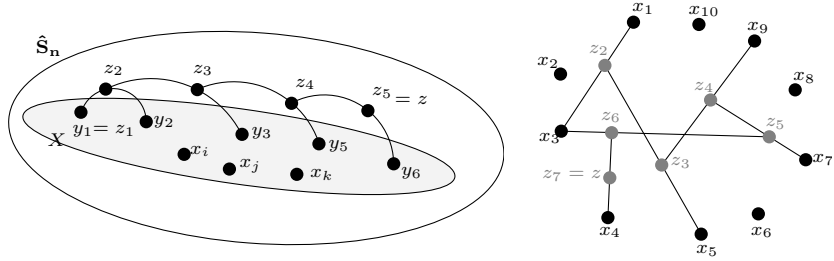
\begin{figure}[ht!]
\begin{scriptsize}
\begin{center}
\begin{tabular}{ll}
\begin{tikzpicture}[line cap=round,line join=round,>=triangle 45,x=0.7cm,y=0.7cm]
\clip(-0.54,-1.8) rectangle (9.08,2.76);
\draw [shift={(1.67,0.78)}] plot[domain=0:2.85,variable=\t]({1*0.63*cos(\t r)+0*0.63*sin(\t r)},{0*0.63*cos(\t r)+1*0.63*sin(\t r)});
\draw [shift={(2.39,-0.33)}] plot[domain=0.47:2.03,variable=\t]({1*1.93*cos(\t r)+0*1.93*sin(\t r)},{0*1.93*cos(\t r)+1*1.93*sin(\t r)});
\draw [shift={(3.96,-0.17)}] plot[domain=0.32:1.98,variable=\t]({1*1.71*cos(\t r)+0*1.71*sin(\t r)},{0*1.71*cos(\t r)+1*1.71*sin(\t r)});
\draw [shift={(5.68,-0.03)}] plot[domain=0:2.06,variable=\t]({1*1.32*cos(\t r)+0*1.32*sin(\t r)},{0*1.32*cos(\t r)+1*1.32*sin(\t r)});
\draw [rotate around={-8.06:(4.06,0.31)},fill=black,fill opacity=0.05] (4.06,0.31) ellipse (2.9cm and 0.5cm);
\draw [rotate around={-4.38:(4.39,0.47)}] (4.39,0.47) ellipse (3.3cm and 1.4cm);
\draw (0.28,0.7) node[anchor=north west] {$X$};
\draw (1.14,0.86) node[anchor=north west] {$=z_1$};
\draw (6.88,1.42) node[anchor=north west] {$=z$};
\draw (0.06,2.56) node[anchor=north west] {$\mathbf{\hat S_n}$};
\fill [color=black] (1.06,0.96) circle (2.5pt);
\draw[color=black] (1.1,0.6) node {$y_1$};
\fill [color=black] (2.3,0.78) circle (2.5pt);
\draw[color=black] (2.64,0.78) node {$y_2$};
\fill [color=black] (4.12,0.54) circle (2.5pt);
\draw[color=black] (4.5,0.56) node {$y_3$};
\fill [color=black] (5.58,0.36) circle (2.5pt);
\draw[color=black] (5.96,0.32) node {$y_5$};
\fill [color=black] (7,-0.02) circle (2.5pt);
\draw[color=black] (7.34,-0.2) node {$y_6$};
\fill [color=black] (1.53,1.4) circle (2.5pt);
\draw[color=black] (1.68,1.76) node {$z_2$};
\fill [color=black] (3.27,1.39) circle (2.5pt);
\draw[color=black] (3.38,1.76) node {$z_3$};
\fill [color=black] (5.06,1.13) circle (2.5pt);
\draw[color=black] (5.32,1.5) node {$z_4$};
\fill [color=black] (3.02,0.16) circle (2.5pt);
\draw[color=black] (3.32,0.22) node {$x_i$};
\fill [color=black] (3.88,-0.12) circle (2.5pt);
\draw[color=black] (4.28,-0.06) node {$x_j$};
\fill [color=black] (5.16,-0.22) circle (2.5pt);
\draw[color=black] (5.64,-0.22) node {$x_k$};
\fill [color=black] (6.51,0.99) circle (2.5pt);
\draw[color=black] (6.78,1.34) node {$z_5$};
\end{tikzpicture}
&
\begin{tikzpicture}[line cap=round,line join=round,>=triangle 45,x=0.5cm,y=0.5cm]
\clip(2.32,-2.88) rectangle (10.2,3.92);
\draw (4.94,3.38)-- (3.02,0.48);
\draw (4.18,2.23)-- (6.74,-2.26);
\draw (5.69,-0.42)-- (8.16,2.88);
\draw (6.99,1.31)-- (9.52,-0.28);
\draw (8.55,0.33)-- (3.02,0.48);
\draw (4.36,0.44)-- (4.26,-1.82);
\fill [color=black] (4.94,3.38) circle (2.5pt);
\draw[color=black] (5.2,3.74) node {$x_1$};
\fill [color=black] (3,2) circle (2.5pt);
\draw[color=black] (2.68,2.3) node {$x_2$};
\fill [color=black] (3.02,0.48) circle (2.5pt);
\draw[color=black] (2.72,0.3) node {$x_3$};
\fill [color=black] (4.26,-1.82) circle (2.5pt);
\draw[color=black] (4.32,-2.2) node {$x_4$};
\fill [color=black] (6.74,-2.26) circle (2.5pt);
\draw[color=black] (6.9,-2.7) node {$x_5$};
\fill [color=black] (8.26,-1.72) circle (2.5pt);
\draw[color=black] (8.36,-2.2) node {$x_6$};
\fill [color=black] (9.52,-0.28) circle (2.5pt);
\draw[color=black] (10,-0.28) node {$x_7$};
\fill [color=black] (9.24,1.58) circle (2.5pt);
\draw[color=black] (9.48,1.9) node {$x_8$};
\fill [color=black] (8.16,2.88) circle (2.5pt);
\draw[color=black] (8.26,3.24) node {$x_9$};
\fill [color=black] (6.68,3.32) circle (2.5pt);
\draw[color=black] (6.82,3.68) node {$x_{10}$};
\fill [color=gray] (4.18,2.23) circle (2.5pt);
\draw[color=gray] (4.02,2.64) node {$z_2$};
\fill [color=gray] (5.69,-0.42) circle (2.5pt);
\draw[color=gray] (6.2,-0.36) node {$z_3$};
\fill [color=gray] (6.99,1.31) circle (2.5pt);
\draw[color=gray] (6.82,1.66) node {$z_4$};
\fill [color=gray] (8.55,0.33) circle (2.5pt);
\draw[color=gray] (8.72,0.68) node {$z_5$};
\fill [color=gray] (4.36,0.44) circle (2.5pt);
\draw[color=gray] (4.34,0.86) node {$z_6$};
\fill [color=gray] (4.31,-0.76) circle (2.5pt);
\draw[color=gray] (3.3,-0.7) node {$z_7=z$};
\end{tikzpicture}
\end{tabular}
\end{center}
\end{scriptsize}
\caption{Accessibility for 10 points: original definition (Figure from~\cite{jam16}).\label{fig2}}
\end{figure}

Here we represent accessible points in a slightly more illustrative way. That is, for a set of given permutation classes $X$, let 
\begin{equation}
\mathcal{Z}(X)=\bigcup\limits_{x,y\in X} \overline{[x,y]}^*.
\end{equation}
Then setting $\mathcal{Z}_0(X):=X$, by induction, we define $\mathcal{Z}_{k+1}(X)=\mathcal{Z}(\mathcal{Z}_k(X))$ for $k\in \N_0$, that is
\[
\mathcal{Z}_{k+1}(X):= \bigcup\limits_{x,y\in \mathcal{Z}_k(X)} \overline{[x,y]}^*.
\]
Finally, the set of all accessible points is defined by
\[
\bar{\mathcal{Z}}(X):=\bigcup\limits_{n\in \N_0} \mathcal{Z}_n(X).
\]
Obviously, these two definitions are not restricted to the case of the symmetric groups and can be considered for a general metric (pseudometric) space $(S,\rho)$. We only need to replace $\hat{S}_n$ and $\overline{[.,.]}^*$ by $S$ and $\overline{[.,.]}_S$, respectively. The latter definition gives a new representation of the former notion of accessibility, and, in fact, we can see that these two definitions are equivalent, in general. We have the following proposition.
\begin{proposition}\label{accessibility equivalence}
Let $S$ be a metric space, then $\bar{\mathcal{Z}}(X)= \bar Z(X)$, for any $X\subset S$.
\end{proposition}
To prove the above proposition, we need the following lemmas.
\begin{lemma}\label{lemma accessible 1}
For $X\subset Y\subset S$, we have $\mathcal{Z}(X)\subset Z(Y)$.
\end{lemma}
\begin{proof}
By definition of $\mathcal{Z}(X)$, it is clear that $\mathcal{Z}(X)\subset \mathcal{Z}(Y)$. Also, if $z\in \mathcal{Z}(Y)$, then there exists $x,y\in Y$ such that $z\in \overline{[x,y]}_S$. Therefore, $z$ is a 1-accessible point of $Y$, i.e. $z\in Z(Y)$. 
\end{proof}
\begin{lemma}\label{lemma accessible 2}
Let $X\subset S$ be such that for any $x,y\in X$, $\overline{[x,y]}_S\subset X$, then $\mathcal{Z}(X)=Z(X)=X$. In particular, $\mathcal{Z}(\bar{\mathcal{Z}}(X))=Z(\bar{\mathcal{Z}}(X))=\bar{\mathcal{Z}}(X)$.
\end{lemma}
\begin{proof}
From the last lemma, it is sufficient to prove $Z(X)=X$, and this is itself clear by definition, as if $z\in Z(X)$, then there exists $m\in \N$, a finite sequence $y_1,..., y_m \in X$, and a finite sequence $z_1=y_1, ..., z_m=z \in S$ such that $z_{i+1} \in \overline{[z_i,y_{i+1}]}_S$. Hence, by assumption, $z_1,...,z_m=z$ are all in $X$. Now, for any two points $x,y\in \bar{\mathcal{Z}}(X)$, there exists $k\in \N$ such that $x,y\in \mathcal{Z}_k(X)$, and thus, $\overline{[x,y]}_S \in \mathcal{Z}_{k+1}(X)\subset \bar{\mathcal{Z}}(X)$. This completes the proof.
\end{proof}

\begin{proof}[Proof of Proposition \ref{accessibility equivalence}]
By Lemma \ref{lemma accessible 1}, $\mathcal{Z}(X)\subset Z(X)$, and hence by induction and applying the same lemma repeatedly, we have $\mathcal{Z}_k(X)\subset Z_k(X)$, for any $k\in \N$. Therefore, $\bar{\mathcal{Z}}(X)\subset \bar{Z}(X)$. To prove the other side, let $z\in Z(X)=Z_1(X)$. There exist $m\in \N$, $y_1,..., y_{m+1} \in X$, and $z_1=y_1, ..., z_{m+1}=z \in S$ such that $z_{i+1} \in \overline{[z_i,y_{i+1}]}_S$. Thus, for $i=2,...,m+1$, $z_i\in \mathcal{Z}_{i-1}(X)$, and in particular, $z\in \mathcal{Z}_{m}(X)$, and hence $Z_1(X)\subset \bar{\mathcal{Z}}(X)$. Now, by Lemma \ref{lemma accessible 1} and Lemma \ref{lemma accessible 2}, 
\[
Z_2(X)=Z(Z_1(X))\subset Z(\bar{\mathcal{Z}}(X))=\bar{\mathcal{Z}}(X).
\]
Repeating this argument, for any $r\in \N_0$, $Z_r(X)\subset \bar{\mathcal{Z}}(X)$ and therefore, $\bar{Z}(X)\subset \bar{\mathcal{Z}}(X)$.
\end{proof}
We say $m\in \N_0$ is the order of accessibility of a set $X$ in a metric space $S$, if $m$ is the minimum number such that $\mathcal{Z}_m(X)=\mathcal{Z}_{m+r}(X)$, for any $r\in \N_0$. In other words, $m$ is the order of accessibility of $X$, if it is the minimum number such that $\bar{\mathcal{Z}}(X)=\mathcal{Z}_m(X)$. If there is no such number, we say the order of accessibility of $X$ is $\infty$. In the case that $S$ is a finite metric space, as in the case of $\hat{S}_n$ endowed with the bp metric, the order of accessibility of any subset of $S$ is finite. As an example, let $A$ be a bounded closed convex subset of $S=\R^n$ (endowed with the Euclidean topology), and let $\partial A$ be its boundary. Then the order of accessibility of $\partial A$ is $1$, and $\bar{Z}(\partial A)=A$.

It is shown in \cite{jam14} that for $k$ permutations $\{x_1,...,x_k\}$ in $S_n$ with maximum distance $n-1$ between any two of them, a permutation $x$ is a median if and only if $\mathcal{A}_x \subset \mathcal{A}_{x_1,...,x_k}$. The situation is similar for $k$ random permutations, since the expected number of common adjacencies for any two random permutations is very small~\cite{jam14}. On the other hand, for $x,y\in S_n$, a permutation $\pi$ lies on $\overline{[x,y]}$ if and only if
\[
\mathcal{A}_{x,y}\subset \mathcal A_{\pi} \subset \mathcal{A}_x \cup \mathcal{A}_y
\]
(see Lemma 2~\cite{jam14}). This shows that the idea of accessible points can be useful in order to find a median far from corners. For example, in the case of three permutations with maximum distance $n-1$ from each other, we can start from two of them, say $x$ and $y$, and find a permutation $\pi \in \overline{[x,y]}$ that is not very close to $x$ and $y$. This must be done by choosing carefully adjacencies from both $x$ and $y$ (including all common adjacencies of them in the case of three random permutations) such that these adjacencies together construct permutation $\pi$. Then we should try to pick some of the adjacencies of $\pi$ (with a sufficient number of adjacencies from each of $x$ and $y$) and also pick some adjacencies from the third permutation, say $z$, to construct a permutation far from all $x,y,z$ whose set of adjacencies is contained in $\mathcal{A}_x\cup \mathcal{A}_y\cup \mathcal{A}_z$, and therefore it is a median of these three points. 

Common adjacencies of permutations can be regarded as a set of segments. A segment (of $S_n$) is a set of consecutive adjacencies of a permutation of length $n$. More explicitly, a segment of length $k\in [n-1]$ is a set of adjacencies 
\[
\{\{n_0,n_1\},\{n_1,n_2\},...,\{n_{k-2},n_{k-1}\},\{n_{k-1},n_k\}\},
\]
where $n_0,n_1,...,n_k\in [n]$ are different natural numbers. It can also be denoted by $[n_0,n_1,...,n_k]$ or equivalently by $[n_k,...,n_1,n_0]$. In particular, any segment of length $n-1$ is the set of adjacencies of a permutation class and vice versa. By convention, we assume that the empty set $\emptyset$ is a segment. We say a segment $s$ is a subsegment of a segment $s'$ if $s\subset s'$. For a given permutation $\pi=\pi_1 \ ... \ \pi_n\in S_n$, for $i\leq j$, the segment $[\pi_i,\pi_{i+1},...,\pi_j]=[\pi_j,...,\pi_{i+1},\pi_i]$ is denoted by $s_{ij}=s_{ij}^\pi$ and is called a segment of $\pi$. We denote by $|s|$ the length of a segment $s$. For a segment $s:=[n_0,...,n_k]$ , $n_0$ and $n_k$ are called end points, and $n_1,...,n_{k-1}$ are called intrinsic points of the segment. Any point (number) which is not either an end point or an intrinsic point of $s$ is called an isolated point with respect to $s$. We denote by $End(s)$, $Int(s)$, and $Iso(s)$, the set of end points, intrinsic points, and isolated points of $s$, respectively. Note that a segment is originally defined as a set of adjacencies and therefore all set operations can be applied on it. Two segments $s=[n_0,...,n_k]$ and $s'=[m_0,...,m_{k'}]$ are said to be strongly disjoint if $\{n_0,...,n_k\}\cap \{m_0,...,m_{k'}\}=\emptyset$. They are disjoint if $s\cap s'=\emptyset$, otherwise we say that they intersect. 

Also, by a set of segments (segment set) of $S_n$, we mean the union of some pairwise strongly disjoint segments of $S_n$. In other words, a set of segments or a segment set $I$ is a subset of $\mathcal{A}_\pi$ for a permutation class $\pi$. In this case, we say $I$ is a segment set of $\pi$ or $\pi$ contains $I$. It is clear that a segment set can be contained in more than one permutation, or in other words, it can be contained in the intersection of adjacencies of several permutations. By a segment (or component) of a segment set $I$ we mean a maximal segment contained in $I$, and to show a segment $s$ is a segment of $I$, we denote $s\hat{\in} I$. Although a segment set $I$ containing segments $s_1,...,s_k$ is in principle the union of adjacencies of $s_i$'s, that is $I=\cup_{i=1}^k s_i$, to ease the notation, we sometime denote it by $\{s_1,...,s_k\}$. Also we denote by $\|I\|:=k$, the number of segments of $I$. Note that the notation $|\ .\ |$ is used for both cardinality of a set and absolute value of a real number. For example, as we already indicated for a segment $s$, $|s|$ is the number of adjacencies of $s$, and,  by the original definition of a segment set as a union of segments, 
\[
|I|=\sum\limits_{i=1}^{\|I\|} |s_i|
\]
is the number of adjacencies of $I$. Also, we frequently use, $|\ .\ |$ for sets such as $End(s)$, $Int(s)$, and $Iso(s)$, to indicate their cardinality. 

Denote by $\mathcal{I}_{m,k}^{(n)}$ the set of all segment sets of $S_n$ with $m$ adjacencies and $k$ segments, i.e. $\mathcal{I}_{m,k}^{(n)}$ is the set of all segment sets $I$ with $|I|=m$ and $\|I\|=k$. Similarly, let $\mathcal{I}_m^{(n)}$ be the set of all segment sets of $S_n$ with $m$ adjacencies. Finally, denote by $\mathcal{I}^{(n)}$, the set of all segment sets of $S_n$. Note that $\mathcal{I}_{m,k}^{(n)}$ may be empty for some $m,k,n$. To have $\mathcal{I}_{m,k}^{(n)}$ non-empty, it is necessary to have
\[
k\leq m\leq n-k,
\]
where the last inequality holds since for a segment set $I\in \mathcal{I}_{m,k}^{(n)}$ and for any arbitrary permutation $\pi$ containing $I$, there should be at least $k-1$ adjacencies of $\pi$ that are not used in $I$ in order to separate $k$ segments of $I$, and therefore $m$ should be bounded by $(n-1)-(k-1)=n-k$. 

It is clear that the intersection of segments (and in general, the intersection of segment sets) is  always a segment set. Two segment sets $I$ and $J$ (in particular two segments $s$ and $s'$, respectively) are said to be consistent, if their union is contained in $\mathcal{A}_\pi$, for a permutation class $\pi$. In particular, any two segment sets of a permutation $\pi$ are consistent. For example, for $n=10$, two segments $[3,7,10,2,5]$ and $[2,5,8,1]$ are consistent and their union is the segment $[3,7,10,2,5,8,1]$, while two segments $[2,6,3,8,1]$ and $[8,1,4,7,6,3,5]$ are not consistent. When we speak of the union of two or more segment sets (respectively, two or more segments) we always assume that they are pairwise consistent. 
We say segment sets $I_1,...,I_k$ complete each other if there exists a permutation $\pi$ such that $\cup_{i=1}^k I_i=\mathcal{A}_\pi$. The complement of a segment set $I$ contained in a permutation $\pi$, is $\bar{I}_\pi:=\mathcal{A}_\pi\setminus I$. In other words, for a segment set $I=\{s_{i_1j_1}^\pi,s_{i_2j_2}^\pi, ...,s_{i_kj_k}^\pi\}$ contained in $\pi$, $\overline{I}_{\pi}=\{s_{j_0i_1}^\pi,s_{j_1i_2}^\pi, ...,s_{j_ki_{k+1}}^\pi\}$, with $j_0=1,i_{k+1}=n$. For $r=1,...,k+1$, we denote by $\overline{I}_{\pi}^{(r)}$ the $r$-th segment of $\overline{I}_{\pi}$ on $\pi$ from left, that is $\overline{I}_{\pi}^{(r)}=s_{j_{r-1}i_r}^\pi$. When we write $\bar{I}_\pi$, we assume that $I$ is contained in $\pi$. 
We can extend the notions of \emph{end point}, \emph{intrinsic point}, and \emph{isolated point} to the case of segment sets as follows. A number $u\in \{1,...,n\}$ is an end point (respectively, an intrinsic point) of a non-empty segment set $I=\{s_1,...,s_l\}$ if it is an end point (respectively, an intrinsic point) of exactly one of the segments of $I$. It is an isolated point of $I$ if it is neither an end point nor an intrinsic point of $I$, or equivalently if it is an isolated point of all of segments of $I$. In other words, using the same notations of $End(I)$, $Int(I)$ and $Iso(I)$ for these three types of points for the segment set $I$, we have
\[
\begin{array}{c}
End(I)=\bigcup\limits_{s\hat{\in} I} End(s)\\
Int(I)=\bigcup\limits_{s\hat{\in} I} Int(s)\\
Iso(I)=\bigcap\limits_{s\hat{\in} I} Iso(s)
\end{array} 
\]
When $I$ is the empty segment set, we define $End(I)=Int(I)=\emptyset$ and $Iso(I)=[n]$. For example, when $n=10$, $I=\{[2,3,9,4],[5,6]\}$ is a segment set having $3,9$ as its intrinsic points, $2,4,5,6$ as its end points, and $1,7,8,10$ as its isolated points. 
We say two segments $s,s'\hat{\in} I$ are neighbours with respect to $\pi$, if there exist $i<j$ such that $\pi_i,\pi_j \in End(s)\cup End(s')$ and for any $k$ with $i<k<j$ (if there is any), $\pi_k\in Iso(I)$. We say a segment $s$ connects two disjoint segments $s_1$ and $s_2$, if $s_1\cup s\cup s_2$ is a segment. 


Given a segment set $I$ of $id$, our goal is to count the number of all permutations $x\in S_n$ such that there exists a permutation $\pi \in \overline{[id,x]}\setminus \{id,x\} \neq \emptyset$ containing $I$ such that $\mathcal{A}_\pi\setminus I$ is a segment set in $x$. In order to find all permutations $x$ with this property, it is convenient to classify the adjacencies of any permutation $x\in S_n$ with respect to $I$. This classification should show all possible ways that every adjacency of $x$ may be used to construct such a permutation $\pi$.  We say the adjacency $\{x_i,x_{i+1}\}$ is \emph{2-free-end}, with respect to $I$, if $x_i$ and $x_{i+1}$ are both isolated points of $I$. It is called \emph{1-free-end}, w.r.t. $I$, if either $x_i$ or $x_{i+1}$ is an isolated point of $I$, and the other is an end point of $I$. It is a \emph{trivial segment}, w.r.t. $I$, if $x_i$ and $x_{i+1}$ are both end points of $I$. Finally, $\{x_1,x_{i+1}\}$ is \emph{0-free-end}, w.r.t. $I$, if either $x_i$ or $x_{i+1}$ is an intrinsic point of $I$. In order to construct a permutation $\pi$ containing $I$ such that $\pi\in\overline{[id,x]}$, $\bar{I}_\pi$ should be contained in $x$. We see that this is an important observation to count the number of permutations $x$ having a permutation $\pi\in \overline{[id,x]}$ far from $id$ and $x$.

\section{Analysis of the adjacency types}
Let $I$ be a segment set of $S_n$. We define $X_n(I)$ to be the set of all permutations $x$ containing a segment set $J$ such that $I\cap J=\emptyset$ and $I\cup J=\mathcal{A}_\pi$ for a permutation $\pi$. Equivalently, letting
\[
\mathbf{C}_n(I)=\{J\in \mathcal{I}^{(n)}: \exists \pi\in S_n \ s.t. \ I\cap J=\emptyset, \ I\cup J=\mathcal{A}_\pi\},
\]
and
\[
\mathcal{R}_n(J)=\{\pi\in S_n: \ J\subset \mathcal{A}_\pi\}, \ for J\in \mathcal{I}^{(n)},
\]
we have
\[
X_n(I)=\bigcup\limits_{J\in \mathbf{C}_n(I)} \mathcal{R}_n(J).
\]
 Note that when $I$ is a segment set of $id$, this definition does not guarantee that $\pi \in \overline{[id,x]}$, and in order to have this property, $\pi$, in addition, must include all adjacencies of $\mathcal{A}_{id,x}$. This motivates us to denote by $\bar{X}_n(I)$ the set of all permutations $x$ for which there exists a permutation $\pi\in \overline{[id,x]}$ whose set of adjacencies can be decomposed into disjoint segment sets $I$ and $J$, i.e. $I\cup J=\mathcal{A}_\pi$, such that $J$ is contained in $x$. In fact, $J$ serves as $\bar{I}_\pi$, the complement of $I$ w.r.t. $\pi$, and by definition $\bar{X}_n(I)\subset X_n(I)$. Therefore, counting the number of elements in $X_n(I)$ gives an upper bound for $|\bar{X}_n(I)|$.
 
To be able to have such a permutation $\pi$, $x$ should contain a segment set $J$ as described above. Then our strategy to count the number of permutations in $X_n(I)$ is, firstly, to find every possible segment set $J$ that can be the complement of $I$ w.r.t. a permutation $\pi$ such that $\mathcal{A}_\pi=I\cup J$, and secondly for each such segment set $J$, to count the number of all possible permutations $x$ containing $J$. Note that for two different segment sets $J$ and $J'$ with the above property, the set of permutations containing $J$ and the set of permutations containing $J'$ do not intersect, and thus considering all possibilities of the segment $J$ gives us a partitioning of $X_n(I)$. An easy observation is that $|\|I\|-\|J\||\leq 1$, and hence there are three possibilities for the number of segments in $J$. On the other hand, all isolated points of $I$, except at most two of them, are intrinsic points of $J$ and vice versa. So this makes it clear how to construct $J$. Basically, depending on the value of $\|I\|-\|J\|$, we take at most two isolated points of $I$ and consider them as end points of $J$. The other end points of $J$ are chosen from the set of end points of $I$, in an appropriate way. Also, the rest of the isolated points of $I$ will be used as intrinsic points of $J$. Once $J$ is determined, it is easy to see that the number of ways one can complete $J$ in order to construct $x$ depends on $\|I\|-\|J\|$ and not $J$ itself. This makes it easy to compute the cardinality of $X_n(I)$, as is indicated in Section \ref{section-main-1}. When it is clear, we drop ``$n$'' from the subscript of $X$.

Motivated by above explanations, to construct a permutation $\pi$ in $\overline{[id,x]}$ containing $I$ such that $\bar{I}_\pi$ is contained in $x$, we cannot use any 0-free-end adjacency of $x$ w.r.t. $I$, since both numbers in the extremities of this type of adjacency are already used in $I$ as its intrinsic points. Therefore, to be able to construct $\pi$ with this property, we must take 2-free-end adjacencies of $x$ w.r.t $I$, to choose the segment set $J$ contained in $x$ as mentioned above. Both the other two types of adjacencies, i.e., 1-free-end and trivial segment adjacencies, can be used only as extremities of segments of $J$. More precisely, a 1-free-end adjacency of $x$ w.r.t $I$ may be used in extremities of segments of any size in $J$, while a trivial segment adjacency of $x$ w.r.t. $I$ may be used only as a segment of length $1$ in $J$. In this section, we compute the expected number (Theorem~\ref{expectation}) and variance (Theorem~\ref{var}) of all four types of adjacencies of a random permutation w.r.t. a random segment contained in $id$, and establish a convergence (in probability) theorem for them. Following this, we study the possibility of constructing a permutation $\pi$ in $\overline{[id,x]}$ containing segment set $I$ from identity such that $\overline{I}_\pi$ is contained in $x$.

The following proposition will be used to prove some of our main results.

\begin{proposition}\label{segmentset}
Given a permutation $x\in S_n$, there exist \[\left(\!\!\!\!\begin{array}{c}
m-1 \\ 
k-1
\end{array}\!\!\!\!\right)\cdot \left(\!\!\!\!\begin{array}{c}
n-m\\ 
k
\end{array}\!\!\!\!\right)\]
segment sets of $x$ with $k>0$ non-empty segments and $m\leq n-1$ adjacencies.
\end{proposition}
\begin{proof}
Consider a segment set $I=\{s_1,...,s_k\}$, with $k$ non-empty segments and $m$ adjacencies that is contained in $x\in S_n$. Then $|\|\overline{I}_{x}\|-k|\leq 1$, and therefore we represent the segments of $\overline{I}_{x}$ by $s'_1,...,s'_{k+1}$, where $s'_j$ is non-empty for $2\leq j\leq k$, and $s'_1$ and $s'_{k+1}$ may be empty. Note that $\sum\limits_{i=1}^{k} |s_i|=m$ and $\sum\limits_{j=1}^{k+1} |s'_j|=n-1-m$ with $|s_i|\geq 1$ for $1\leq i\leq k$ and $|s'_j|\geq 1$ for $2\leq j\leq k$. Hence, the number of solutions for these two equations is equal to:

 \[\left(\!\begin{array}{c}
m-k+(k-1) \\ 
k-1
\end{array}\!\right)\cdot \left(\!\!\!\!\begin{array}{c}
n-1-m-(k-1)+(k+1-1)\\ 
(k+1-1)
\end{array}\!\!\!\!\right)=\left(\!\!\!\!\begin{array}{c}
m-1 \\ 
k-1
\end{array}\!\!\!\!\right)\cdot \left(\!\!\!\!\begin{array}{c}
n-m\\ 
k
\end{array}\!\!\!\!\right)\]

In other words, that is the number of ways we can choose $k$ segments with $m$ adjacencies of $x$. 
\end{proof}

We assume that all random elements and variables are defined on a probability space $(\Omega, \p, \mathcal{F})$, and denote by $\e[\ .\ ]$ and $Var(\ .\ )$, the expected value and variance of a random variable, respectively. We denote by $\xi^{(n)}$, a permutation chosen uniformly at random from $S_n$, and by $I_m^{(n)}$ a segment set chosen uniformly at random from 
$\mathcal{I}_m^{(n)}$. Similarly, let us denote by $I_{m,k}^{(n)}$ a segment set 
chosen uniformly at random from 
$\mathcal{I}_{m,k}^{(n)}$, and let $A_{m,k}^{(n)}$ be the event that $I_m^{(n)}$ has $k$ segments, that is $A_{m,k}^{(n)}:=\{I_m^{(n)}\in \mathcal{I}_{m,k}^{(n)}\}$. We also assume that $\xi^{(n)}$, $I_m^{(n)}$, and $I_{m,k}^{(n)}$ are independent. Let $\alpha ,\beta , \gamma , \delta$ be functions
\[
\alpha ,\beta , \gamma , \delta: \bigcup\limits_{n\in \N} (S_n\times \mathcal{I}^{(n)}) \rightarrow \N_0,
\]
such that, for $x\in S_n$ and a segment set of $S_n$, namely $I$, let $\alpha(x,I)$, $\beta(x,I)$, $\gamma(x,I)$ and $\delta(x,I)$ be the number of 2-free-end adjacencies, 1-free-end adjacencies, trivial segments, and 0-free-end adjacencies of $x$ w.r.t. $I$, respectively. In particular, let $\alpha_m^{(n)}:=\alpha(\xi^{(n)},I_m^{(n)})$, $\beta_m^{(n)}:=\beta(\xi^{(n)},I_m^{(n)})$, $\gamma_m^{(n)}:=\gamma(\xi^{(n)},I_m^{(n)})$ and $\delta_m^{(n)}:=\delta(\xi^{(n)},I_m^{(n)})$. Similarly, let $\alpha_{m,k}^{(n)}:=\alpha(\xi^{(n)},I_{m,k}^{(n)})$, $\beta_{m,k}^{(n)}:=\beta(\xi^{(n)},I_{m,k}^{(n)})$, $\gamma_{m,k}^{(n)}:=\gamma(\xi^{(n)},I_{m,k}^{(n)})$ and $\delta_{m,k}^{(n)}:=\delta(\xi^{(n)},I_{m,k}^{(n)})$. When there is no risk of confusion, we drop ``$n$'' from the superscripts.
\begin{theorem}\label{expectation}
Let $m=m(n)$ and $k=k(n)$ be such that  $0< k\leq m<n$, and let $I$ be an arbitrary segment set in $\mathcal{I}_{m,k}$. Then
\begin{equation}\label{conditional expected number}
\begin{array}{l}
{\displaystyle \e[\alpha_m |A_{m,k}]=\e[\alpha_{m,k}]=\e[\alpha(\xi,I)]=\frac{(n-m-k)(n-m-k-1)}{n},}\\\\
{\displaystyle \e[\beta_m |A_{m,k}]=\e[\beta_{m,k}]=\e[\beta(\xi,I)]=\frac{4k(n-m-k)}{n},}\\\\
{\displaystyle \e[\gamma_m |A_{m,k}]=\e[\gamma_{m,k}]=\e[\gamma(\xi,I)]=\frac{2k(2k-1)}{n}},\\\\
{\displaystyle \e[\delta_m |A_{m,k}]=\e[\delta_{m,k}]=\e[\delta(\xi,I)]=\frac{(m-k)(2n-m+k-1)}{n}.}
\end{array}
\end{equation}
Furthermore,
\begin{equation}\label{expected number}
\begin{array}{l}
{\displaystyle \e[\alpha_m]=\frac{(n- m) (n-m-1)^2 (n-m-2)}{n(n-1)(n-2)},}\\\\
{\displaystyle \e[\beta_m]=\frac{4 m (n-m) (n-m-1)^2}{n(n-1)(n-2)},}\\\\
{\displaystyle \e[\gamma_m]=\frac{2 m (n-m) (2 m (n-m) + n)}{n(n-1)(n-2)}},\\\\
{\displaystyle \e[\delta_m]=\frac{m(m-1)(2 n^2- 6 n- m^2+ 3 m  +2)}{n (n-1)(n-2)}.}
\end{array}
\end{equation}


\end{theorem}

\begin{proof}
For $i=1,...,n-1$, let $\hat\alpha_{m,i}$, $\hat\beta_{m,i}$, $\hat\gamma_{m,i}$ and $\hat\delta_{m,i}$ be random variables such that $\hat\alpha_{m,i}=1$ if the $i$-th adjacency of $\xi$, i.e. $\{\xi_i,\xi_{i+1}\}$, is 2-free-end w.r.t. $I_m$ and $\hat\alpha_{m,i}=0$ otherwise; $\hat\beta_{m,i}=1$ if the $i$-th adjacency of $\xi$ is 1-free-end w.r.t. $I_m$ and $\hat\beta_{m,i}=0$ otherwise; $\hat\gamma_{m,i}=1$ if the $i$-th adjacency of $\xi$ is a trivial segment and $\hat\gamma_{m,i}=0$ otherwise; and $\hat\delta_{m,i}=1$ if the $i$-th adjacency of $\xi$ is 0-free-end and $\hat\delta_{m,i}=0$ otherwise. Then, for every $i=1,...,n-1$, we have:
\begin{equation*}
\begin{array}{l}
{\displaystyle \p(\hat\alpha_{m,i}=1|A_{m,k})=\p(\hat \alpha_{m,i}=1|I_m=I)= \frac{(n-m-k)(n-m-k-1)}{n(n-1)},}\\\\
{\displaystyle \p(\hat \beta_{m,i}=1|A_{m,k})=\p(\hat \beta_{m,i}=1|I_m=I)=\frac{4k(n-m-k)}{n(n-1)},}\\\\
{\displaystyle \p(\hat \gamma_{m,i}=1|A_{m,k})=\p(\hat \gamma_{m,i}=1|I_m=I)=\frac{2k(2k-1)}{n(n-1)},}\\\\
{\displaystyle \p(\hat \delta_{m,i}=1|A_{m,k})=\p(\hat \delta_{m,i}=1|I_m=I)=\frac{(m-k)(2n-m+k-1)}{n(n-1)}.}
\end{array}
\end{equation*}
Therefore,
\begin{multline*}
\e[\alpha_m |A_{m,k}]=\e[\alpha_{m,k}]=\e[\alpha(\xi,I)]=\\
\sum\limits_{i=1}^{n-1}\p(\hat\alpha_{m,i}=1 |A_{m,k})=\frac{(n-m-k)(n-m-k-1)}{n}.
\end{multline*}
The other conditional expected values of (\ref{conditional expected number}) are proved similarly.

From Proposition~\ref{segmentset}, the probability that $A_{m,k}$ occurs is 

\[ \p(A_{m,k})=\frac{\left(\begin{array}{c}
m-1 \\ 
k-1
\end{array}\right) \left(\begin{array}{c}
n-m\\ 
k
\end{array}\right)}{\left(\begin{array}{c}
n-1 \\ 
m
\end{array}\right)}.\]

Therefore, by averaging over $k$, we have

\begin{flalign*}
\e[\alpha_m]&=\sum\limits_{k=1}^{m}\frac{(n-m-k)(n-m-k-1)}{n}\frac{\left(\begin{array}{c}
m-1 \\ 
k-1
\end{array}\right) \left(\begin{array}{c}
n-m\\ 
k
\end{array}\right)}{\left(\begin{array}{c}
n-1 \\ 
m
\end{array}\right)}&\\
&=\frac{(n- m) (1 + m - n)^2 (n - m-2)}{(n-2) (n-1) n},&
\end{flalign*}

\begin{flalign*}
 \e[\beta_m]&=\sum\limits_{k=1}^{m}\frac{4k(n-m-k)}{n}\frac{\left(\begin{array}{c}
m-1 \\ 
k-1
\end{array}\right) \left(\begin{array}{c}
n-m\\ 
k
\end{array}\right)}{\left(\begin{array}{c}
n-1 \\ 
m
\end{array}\right)}
=\frac{4 m (n-m) (1 + m - n)^2}{(n-2) (n-1) n},&
\end{flalign*}

\begin{flalign*}
\e[\gamma_m]&=\sum\limits_{k=1}^{m}\frac{2k(2k-1)}{n}\frac{\left(\begin{array}{c}
m-1 \\ 
k-1
\end{array}\right) \left(\begin{array}{c}
n-m\\ 
k
\end{array}\right)}{\left(\begin{array}{c}
n-1 \\ 
m
\end{array}\right)}=\frac{2 m (n-m) (2 m (n-m) + n)}{(n-2) (n-1) n}, &
\end{flalign*}
and
\begin{flalign*}
 \e[\delta_m]&=\sum\limits_{k=1}^{m}\frac{(m-k)(2n-m+k-1)}{n}\frac{\left(\begin{array}{c}
m-1 \\ 
k-1
\end{array}\right) \left(\begin{array}{c}
n-m\\ 
k
\end{array}\right)}{\left(\begin{array}{c}
n-1 \\ 
m
\end{array}\right)}&\\
&=\frac{m(m-1)(2 n^2- 6 n- m^2+ 3 m  +2)}{n (2 - 3 n + n^2)}.&
\end{flalign*}

\end{proof}


\begin{theorem}\label{var}
Let $m=m(n)$ and $k=k(n)$ be such that $0< k\leq m<n$, and let $I$ be an arbitrary segment set in $\mathcal{I}_{m,k}$. Then
\begin{footnotesize}
\begin{flalign*}
Var(\alpha_{m,k})=Var(\alpha(\xi,I))&=\e[\alpha_{m,k}](1-\e[\alpha_{m,k}])+\frac{(n-m-k)(n-m-k-1)^2(n-m-k-2)}{n(n-1)}\\
&=(1-\frac{m+k}{n})^2(\frac{m+k}{n})^2n+o(n),
\end{flalign*}
\begin{flalign*}
Var(\beta_{m,k})=&Var(\beta(\xi,I))=\e[\beta_{m,k}](1-\e[\beta_{m,k}])+\frac{4k(n-m-k)((n-m-k-1)(4k-1)+2k-1)}{n(n-1)}&
\end{flalign*}
\begin{flalign*}
&=4\frac{k}{n}(1-\frac{m+k}{n})(\frac{k}{n}(3-\frac{4k}{n})+\frac{m}{n}(1-\frac{4k}{n}))n+o(n),
\end{flalign*}

\begin{flalign*}
Var(\gamma_{m,k})=&Var(\gamma(\xi,I))=\e[\gamma_{m,k}](1-\e[\gamma_{m,k}])+\frac{2k(2k-1)^2(2k-2)}{n(n-1)}=4(1-\frac{2k}{n})^2(\frac{k}{n})^2n+o(n),&
\end{flalign*}

\begin{flalign*}
Var(\delta_{m,k})=&Var(\delta(\xi,I))=\e[\delta_{m,k}](1-\e[\delta_{m,k}])&
\end{flalign*}

\begin{flalign*}
&+\frac{(m-k)[(m-k-1)(2n-m+k-2)(2n-m+k-3)+2(n-2)(n-1)]}{2n(n-1)}\\
&=(\frac{m-k}{n})^2(1-\frac{m-k}{n})^2n+o(n).
\end{flalign*}
Furthermore,


\begin{multline*}
Var(\alpha_{m})=\e[\alpha_{m}](1-\e[\alpha_{m}])+\sum\limits_{k=1}^{m}\frac{(n-m-k)(n-m-k-1)^2(n-m-k-2)}{n(n-1)}\p(A_{m,k})=\\ (1-\frac{m}{n})^4(\frac{m}{n})^2 (8+\frac{m}{n}(-12+\frac{5m}{n}))n+o(n),
\end{multline*}

\begin{multline*}
Var(\beta_{m})=\e[\beta_{m}](1-\e[\beta_{m}])+\sum\limits_{k=1}^{m}\frac{4k(n-m-k)((n-m-k-1)(4k-1)+2k-1)}{n(n-1)}\p(A_{m,k})=\\
4(1-\frac{m}{n})^3(\frac{m}{n})^2(8-\frac{m}{n}(31+\frac{4m}{n}(-11+\frac{5m}{n})))n+o(n),
\end{multline*}

\begin{multline*}
Var(\gamma_{m})=\e[\gamma_{m}](1-\e[\gamma_{m}])+\sum\limits_{k=1}^{m}\frac{2k(2k-1)^2(2k-2)}{n(n-1)}\p(A_{m,k})=\\
4(1-\frac{m}{n})^2(\frac{m}{n})^2(1-4(1-\frac{m}{n})(\frac{m}{n})(1+5(1-\frac{m}{n})\frac{m}{n}))n+o(n),
\end{multline*}


\begin{multline*}
Var(\delta_{m})=\e[\delta_{m}](1-\e[\delta_{m}])\\
+\sum\limits_{k=1}^{m}\frac{(m-k)[(m-k-1)(2n-m+k-2)(2n-m+k-3)+2(n-2)(n-1)]}{2n(n-1)}\p(A_{m,k})=\\
(\frac{m}{n})^2(1-(\frac{m}{n})^2)^2(4+\frac{m}{n}(-8+\frac{5m}{n}))n+o(n),
\end{multline*}

\begin{flalign*}
where &~~~~\p(A_{m,k})=\frac{\left(\begin{array}{c}
m-1 \\ 
k-1
\end{array}\right) \left(\begin{array}{c}
n-m\\ 
k
\end{array}\right)}{\left(\begin{array}{c}
n-1 \\ 
m
\end{array}\right)}.&
\end{flalign*}

\end{footnotesize}
\end{theorem}

\begin{proof}
For $i=1,...,n-1$, recall the definition of $\hat\alpha_{m,i}$, $\hat\beta_{m,i}$, $\hat\gamma_{m,i}$ and $\hat\delta_{m,i}$ from the proof of Theorem \ref{expectation}, and similarly, let $\hat\alpha_{m,k,i}$, $\hat\beta_{m,k,i}$, $\hat\gamma_{m,k,i}$ and $\hat\delta_{m,k,i}$ be random variables such that $\hat\alpha_{m,k,i}=1$ if the $i$-th adjacency of $\xi$, i.e. $\{\xi_i,\xi_{i+1}\}$, is 2-free-end w.r.t. $I_{m,k}$ and $\hat\alpha_{m,k,i}=0$ otherwise; $\hat\beta_{m,k,i}=1$ if the $i$-th adjacency of $\xi$ is 1-free-end w.r.t. $I_{m,k}$ and $\hat\beta_{m,k,i}=0$ otherwise; $\hat\gamma_{m,k,i}=1$ if the $i$-th adjacency of $\xi$ is trivial segment w.r.t. $I_{m,k}$ and $\hat\gamma_{m,k,i}=0$ otherwise; and $\hat\delta_{m,k,i}=1$ if the $i$-th adjacency of $\xi$ is 0-free-end w.r.t. $I_{m,k}$ and $\hat\delta_{m,k,i}=0$ otherwise. Then, for every $i=1,...,n-1$, we have:
\[
\begin{array}{l}
\e[\alpha_{m,k}^2]=\sum\limits_i\e[\hat\alpha_{m,k,i}^2]+2\sum\limits_{i> j}\e[\hat\alpha_{m,k,i} \hat\alpha_{m,k,j}]=\\
\sum\limits_i \p(\hat\alpha_{m,k,i}^2=1)+2\sum\limits_{i> j}\p(\hat\alpha_{m,k,i} \hat\alpha_{m,k,j}=1)=\\
\sum\limits_i \p(\hat \alpha_{m,k,i}=1)+2\sum\limits_{i> j}\p(\hat \alpha_{m,k,i} \hat\alpha_{m,k,j}=1)=\\
\e[\alpha_{m,k}]+2\sum\limits_{i> j}\p(\hat \alpha_{m,k,i} \hat\alpha_{m,k,j}=1).
\end{array}
\]
Now, note that:

\begin{small}
\[
\begin{array}{l}
{\displaystyle \sum\limits_{i>j+1}\p(\hat \alpha_{m,k,i} \hat\alpha_{m,k,j}=1)=\sum\limits_{i>j+1}\frac{(n-m-k)(n-m-k-1)(n-m-k-2)(n-m-k-3)}{n(n-1)(n-2)(n-3)}}\\
{\displaystyle =\frac{(n-m-k)(n-m-k-1)(n-m-k-2)(n-m-k-3)}{2n(n-1)}},
\end{array}
\]
\end{small}
and
\begin{small}
\[
\begin{array}{l}
{\displaystyle \sum\limits_{i=j+1}\p(\hat \alpha_{m,k,i} \hat\alpha_{m,k,j}=1)=\sum\limits_{i=j+1}\frac{(n-m-k)(n-m-k-1)(n-m-k-2)}{n(n-1)(n-2)}}\\
{\displaystyle =\frac{(n-m-k)(n-m-k-1)(n-m-k-2)}{n(n-1)}}.
\end{array}
\]
\end{small}
Hence,
\begin{small}
\[
\begin{array}{l}
{\displaystyle Var(\alpha_{m,k})=\e[\alpha_{m,k}^2]-(\e[\alpha_{m,k}])^2}\\
{\displaystyle =\e[\alpha_{m,k}](1-\e[\alpha_{m,k}])+\frac{(n-m-k)(n-m-k-1)^2(n-m-k-2)}{n(n-1)}.}
\end{array}
\]
\end{small}
Exactly the same calculations give $Var(\alpha(\xi,I))$. Similarly we can compute $Var(\beta_{m,k})=Var(\beta(\xi,I))$, $Var(\gamma_{m,k})=Var(\gamma(\xi,I))$ and $Var(\delta_{m,k})=Var(\delta(\xi,I))$.\\

\noindent Now to compute $Var(\alpha_m)$ write
\[
\begin{array}{l}
\e[\alpha_m^2]=\sum\limits_i\e[\hat\alpha_{m,i}^2]+2\sum\limits_{i> j}\e[\hat\alpha_{m,i} \hat\alpha_{m,j}]=\sum\limits_i \p(\hat \alpha_{m,i}^2=1)+2\sum\limits_{i> j}\p(\hat \alpha_{m,i} \hat \alpha_{m,j}=1)=\\
\sum\limits_i \p(\hat \alpha_{m,i}=1)+2\sum\limits_{i> j}\p(\hat \alpha_{m,i} \hat \alpha_{m,j}=1)=\e[\alpha]+2\sum\limits_{i> j}\p(\hat \alpha_{m,i} \hat \alpha_{m,j}=1).
\end{array}
\]
Now, we note that:

\begin{footnotesize}
\begin{flalign*}
\sum\limits_{i>j+1}\p(\hat \alpha_{m,i}\cdot\hat  \alpha_{m,j}=1)&=&
\end{flalign*}

\begin{flalign*}
&\sum\limits_{i>j+1}\sum\limits_{k=1}^{m}\frac{(n-m-k)(n-m-k-1)(n-m-k-2)(n-m-k-3)}{n(n-1)(n-2)(n-3)}\p(A_{m,k})&
\end{flalign*}

\begin{flalign*}
&=\sum\limits_{k=1}^{m}\frac{(n-m-k)(n-m-k-1)(n-m-k-2)(n-m-k-3)}{2n(n-1)}\p(A_{m,k}),&
\end{flalign*}

\end{footnotesize}

\noindent and

\begin{footnotesize}
\begin{flalign*}
\sum\limits_{i=j+1}\p(\hat \alpha_{m,i} \hat \alpha_{m,j}=1)&=\sum\limits_{i=j+1}\sum\limits_{k=1}^{m}\frac{(n-m-k)(n-m-k-1)(n-m-k-2)}{n(n-1)(n-2)}\p(A_{m,k})&\\
&=\sum\limits_{k=1}^{m}\frac{(n-m-k)(n-m-k-1)(n-m-k-2)}{n(n-1)}\p(A_{m,k}).&
\end{flalign*}
\end{footnotesize}



\noindent Therefore,

\begin{footnotesize}
\begin{flalign*}
~~~~&Var(\alpha_m)=\e[\alpha_m^2]-(\e[\alpha_m])^2&\\
&=\e[\alpha_m](1-\e[\alpha_m])+\sum\limits_{k=1}^{m}\frac{(n-m-k)(n-m-k-1)^2(n-m-k-2)}{n(n-1)}\p(A_{m,k})&\\
&=\e[\alpha_m](1-\e[\alpha_m])&\\
&+\frac{(n-m) (n-m-1)^2 (n-m-2)^2 (n-m-3) \left(n^2-5 n+4-2 m n+m (m+7)\right)}{n(n-1)^2(n-2)(n-3)(n-4)}&\\
&=(1-\frac{m}{n})^4(\frac{m}{n})^2 (8+\frac{m}{n}(-12+5\frac{m}{n}))n+o(n).&
\end{flalign*}
\end{footnotesize}
\noindent Similarly we can show that

\begin{footnotesize}
\begin{flalign*}
~~~~&Var(\beta_m)=\e[\beta_m^2]-(\e[\beta_m])^2&\\
&=\e[\beta_m](1-\e[\beta_m])+\sum\limits_{k=1}^{m}\frac{4k(n-m-k)((n-m-k-1)(4k-1)+2k-1)}{n(n-1)}\p(A_{m,k})&\\
&=\e[\beta_m](1-\e[\beta_m])+\left(\frac{4 m (m-n) (m-n+1)^2}{(n-4) (n-3) (n-2)(n-1)^2 n}\right)\times &\\
& \left((1-4 m) n^3+(4 m (3 m+5)-3) n^2-(m+1) (3 m (4 m+11)+1) n+4 (m+1)^2 (m (m+4)+1)\right) &\\
&=4(1-\frac{m}{n})^3(\frac{m}{n})^2(8-\frac{m}{n}(31+4\frac{m}{n}(-11+5\frac{m}{n})))n+o(n),&
\end{flalign*}
\end{footnotesize}

\begin{footnotesize}
\begin{flalign*}
~~~~&Var(\gamma_m)=\e[\gamma_m^2]-(\e[\gamma_m])^2&\\
&=\e[\gamma_m](1-\e[\gamma_m])+\sum\limits_{k=1}^{m}\frac{2k(2k-1)^2(2k-2)}{n(n-1)}\p(A_{m,k})=&\\
&=\e[\gamma_m](1-\e[\gamma_m])&\\
&+\frac{4 (m-1) m (m-n) (m-n+1) \left(4 m^4-8 m^3 n+4 m^2 \left(n^2+n+3\right)-4 m n (n+3)+n (n+9)-4\right)}{(n-4) (n-3) (n-2)(n-1)^2 n}&\\
&=4(1-\frac{m}{n})^2(\frac{m}{n})^2(1-4(1-\frac{m}{n})(\frac{m}{n})(1+5(1-\frac{m}{n})\frac{m}{n}))n+o(n),&
\end{flalign*}
\end{footnotesize}
and finally,

\begin{footnotesize}
\begin{flalign*}
&Var(\delta_m)=\e[\delta_m^2]-(\e[\delta_m])^2&\\
&=\e[\delta_m](1-\e[\delta_m])&\\
&+\sum\limits_{k=1}^{m}\frac{(m-k)[(m-k-1)(2n-m+k-2)(2n-m+k-3)+2(n-2)(n-1)]}{2n(n-1)}\p(A_{m,k})&\\
&=\e[\delta_m](1-\e[\delta_m])&\\
&+\frac{(m-1)m}{(n-4) (n-3) (n-2) (n-1)^2 n} \times \left\lbrace (m-5) m \left(m \left(m^3-10 m^2+m+40\right)+4\right)+4 (m-4) (m+1) n^4 \right. &\\
&\left. +2 (9-23 (m-3) m) n^3+2 (m (m (51-2 (m-8) m)-235)+50) n^2\right. &\\
&\left. +2 m (m (13 (m-8) m+121)+170) n+2 n^5-152 n+48\right\rbrace &\\
&=(\frac{m}{n})^2(1-(\frac{m}{n})^2)^2(4+\frac{m}{n}(-8+\frac{5m}{n}))n+o(n).&
\end{flalign*}
\end{footnotesize}
\end{proof}

We are ready to state a convergence theorem for all different types of adjacencies of $\xi^{(n)}$ w.r.t. $I_{m(n),k(n)}^{(n)}$ or $I_{m(n)}^{(n)}$. Let $m:\N\to \N$ and $k:\N\to \N$ be such that $1\leq k(n)\leq m(n)\leq n-k(n)$, for any $n\in \N$. Also, let $(\hat{I}_n)_{n\in \N}$ be an arbitrary sequence of segment sets that $\hat{I}_n\in \mathcal{I}_{m(n),k(n)}^{(n)}$. Denote
\[
\begin{array}{l}
\tilde{\alpha}_n:=\alpha(\xi^{(n)},I_{m(n)}^{(n)}), \ and\\
\bar{\alpha}_n:=\alpha(\xi^{(n)},I_{m(n),k(n)}^{(n)}).\\
\end{array}
\]
Similarly, for $n\in \N$, we define $\tilde \beta_n, \tilde \gamma_n, \tilde \delta_n$, and $\bar \beta_n, \bar \gamma_n, \bar \delta_n$.
\begin{theorem}\label{convergence in probability}
Suppose $\frac{m(n)}{n}\to c$ and $\frac{k(n)}{n}\to c'$, as $n \to \infty$. Then, as $n\to \infty$
\begin{footnotesize}
\[
\begin{array}{l}
{\displaystyle \frac{\tilde{\alpha}_n}{n}\overset{L^2,p}{\longrightarrow} (1-c)^4}, \\\\
{\displaystyle\frac{\tilde{\beta}_n}{n}\overset{L^2,p}{\longrightarrow} 4c(1-c)^3},\\\\
{\displaystyle \frac{\tilde{\gamma}_n}{n}\overset{L^2,p}{\longrightarrow} 4c^2(1-c)^2},\\\\
{\displaystyle \frac{\tilde{\delta}_n}{n}\overset{L^2,p}{\longrightarrow} c^2(2-c)^2},\\\\
{\displaystyle \frac{\bar{\alpha}_n}{n} \ , \ \frac{\alpha (\xi^{(n)},\hat{I}_n)}{n}\overset{L^2,p}{\longrightarrow} (1-c-c')^2},\\\\
{\displaystyle \frac{\bar{\beta}_n}{n} \ , \ \frac{\beta (\xi^{(n)},\hat{I}_n)}{n}\overset{L^2,p}{\longrightarrow} 4c'(1-c-c')},\\\\
{\displaystyle \frac{\bar{\gamma}_n}{n} \ , \ \frac{\gamma (\xi^{(n)},\hat{I}_n)}{n}\overset{L^2,p}{\longrightarrow} 4c'^2},\\\\
{\displaystyle \frac{\bar{\delta}_n}{n} \ , \ \frac{\delta (\xi^{(n)},\hat{I}_n)}{n}\overset{L^2,p}{\longrightarrow} (c-c')(2-c+c')}.\\\\
\end{array}
\]
\end{footnotesize}


\end{theorem}
\begin{proof}
First observe that, by Theorem \ref{expectation}, as $n\to \infty$,
\[
\begin{array}{l}
{\displaystyle \e[\frac{\tilde{\alpha}_n}{n}]\to (1-c)^4}, \\\\
{\displaystyle \e[\frac{\tilde{\beta}_n}{n}]\to 4c(1-c)^3},\\\\
{\displaystyle \e[\frac{\tilde{\gamma}_n}{n}]\to 4c^2(1-c)^2},\\\\
{\displaystyle \e[\frac{\tilde{\delta}_n}{n}]\to c^2(2-c)^2},\\\\
{\displaystyle \e[\frac{\bar{\alpha}_n}{n}] \ , \ \e[\frac{\alpha (\xi^{(n)},\hat{I}_n)}{n}]\to (1-c-c')^2},\\\\
{\displaystyle \e[\frac{\bar{\beta}_n}{n}] \ , \ \e[\frac{\beta (\xi^{(n)},\hat{I}_n)}{n}]\to 4c'(1-c-c')},\\\\
{\displaystyle \e[\frac{\bar{\gamma}_n}{n}] \ , \ \e[\frac{\gamma (\xi^{(n)},\hat{I}_n)}{n}]\to 4c'^2},\\\\
{\displaystyle \e[\frac{\bar{\delta}_n}{n}] \ , \ \e[\frac{\delta (\xi^{(n)},\hat{I}_n)}{n}]\to (c-c')(2-c+c')}.\\\\
\end{array}
\]
Also, following Theorem \ref{var}, the variances of all these sequences converge to $0$. Hence, the convergence in $L^2$ and in probability holds.
\end{proof}

Let $I$ be a segment set of $id^{(n)}$. In order to construct a permutation $x\in X_n(I)$, we need to find a segment set of $S_n$, namely $J$, such that $I\cap J=\emptyset$ and $I\cup J=\mathcal{A}_\pi$, for a permutation $\pi$. Then, $x$ is constructed by completing the segment set $J$. Conversely, when a permutation $x\in X_n(I)$ is given, an easy observation shows that there exists at least one permutation $\pi$ containing $I$ such that $J=\mathcal{A}_\pi\setminus I \subset \mathcal{A}_x$ and all 2-free-end adjacencies of $x$ are used in $\pi$ (Lemma \ref{lemma-J-interior}). For the moment, let us denote by $J\strut^\mathrm{o}$, the segment set of $x$ containing all 2-free-end adjacencies of $x$ w.r.t. $I$, and note that we must have $J\strut^\mathrm{o}\subset J$. So in order to find the permutation $\pi$ with the above property, we first take the segment set $I\cup J\strut^\mathrm{o}$. In fact, $\mathcal{A}_\pi \setminus (I\cup J\strut^\mathrm{o})$ should still be a segment set of $x$, and $n-1-|I|-|J\strut^\mathrm{o}|$ more adjacencies of $x$ (1-free-end adjacencies and trivial segments) should be taken in order to complete $I\cup J\strut^\mathrm{o}$. To analyse this further, we define this more formally as follows. Let $F$ be a function
\[
F:\bigcup\limits_{n\in \N}(S_n\times \mathcal{I}^{(n)})\rightarrow \mathcal{I}^{(n)},
\]
where for any permutation $x\in S_n$ and any segment set $I\in \mathcal{I}^{(n)}$, $F(x,I)$ is the segment set of $x$ containing all 2-free-end adjacencies of $x$ w.r.t. $I$, that is
\[
F(x,I):=\{\{l,l'\}\in \mathcal{A}_x : \ \{l,l'\} \ is \ 2-free-end\}.
\]
Let
\[
Q:\bigcup\limits_{n\in \N}(S_n\times \mathcal{I}^{(n)})\rightarrow \N_0,
\]
where for $(x,I)\in S_n\times \mathcal{I}^{(n)}$, $Q(x,I)$ is the number of adjacencies needed in order to complete $I\cup F(x,I)$ to a permutation $\pi$, that is
\[
Q(x,I)=n-1-|I|-|F(x,I)|.
\]
The following theorem restricts the range of $Q(x,I)$, for $x\in X_n(I)$.
\begin{theorem}\label{interval-theorem}
Let $I\in \mathcal{I}^{(n)}$, and $x\in X_n(I)$. Then
\[
\|I\|-1\leq Q(x,I)\leq 2\|I\|.
\]
\end{theorem}
Before proving the above theorem, we introduce a new concept. Let $I$ be a segment set of $S_n$. The freedom factor of a point (number) $k\in [n]$, is $0$ if $k\in Int(I)$. It is $1$, if $k\in End(I)$. Finally, it is $2$, if $k\in Iso(E)$. Similarly,  the freedom factor of a segment $s=[v_1,...,v_l]=[v_l,...,v_1]$ is denoted by $u=<u_1,...,u_l>=<u_l,...,u_1>$, where for each $i\in [l]$, $u_i$ is the freedom factor of $v_i$. A segment $s$, with the freedom vector $u$ is called a $u$-segment. Also, for $\pi\in S_n$ and $i\in [n]$, the set of neighbours of $i$ in $\pi$ is defined by
\[
\mathcal{N}_\pi(i):=\{j\in [n]: \ \{i,j\}\in \A_\pi\}.
\]
For an arbitrary segment set of $S_n$, namely $I$, in order that $x\in X_n(I)$, we need to find a segment set $J$ contained in $x$ such that $I\cup J=\mathcal{A}_\pi$ and $I\cap J=\emptyset$. As we mentioned, $J$ may not have all adjacencies of $F(x,I)$. For instance, let $I=[4,5,6,7]$ and $x=6 \ 4 \ 1 \ 3 \ 8 \ 10 \ 2 \ 9 \ 7 \ 5$. Then $x\in X_{10}(I)$ and $J_1=\{[3,1,4],[7,9,2,10,8]\}$ have the required property, while it does not contain the adjacency $\{3,8\} \in F(x,I)$. However, even in this case, we see that there are segment sets $J_2=[9,2,10,8,3,1,4]$ and $J_3=[1,3,8,10,2,9,7]$ including all adjacencies of $F(x,I)$ both with the required properties. In fact, in the following lemma we can see that there are not many adjacencies of $F(x,I)$ that can be ignored in the construction of $\pi$ from $x$ and $I$.
\begin{lemma}\label{lemma-J-interior}
Let $I$ be a segment set of $\mathcal{I}^{(n)}$, and $x\in X_n(I)$.
\begin{itemize}
\item[a)] Let $\pi\in S_n$ be such that $I\subset \mathcal{A}_\pi$, and $\mathcal{A}_\pi\setminus I \subset \mathcal{A}_x$. Then either $F(x,I)\subset \mathcal{A}_\pi \setminus I$, or there exists an adjacency of $F(x,I)$, namely $e \in F(x,I)$ such that $F(x,I)\setminus \{e\} \subset \mathcal{A}_\pi\setminus I$.
\item[b)] There always exists a permutation $\pi\in S_n$ such that $I\subset \mathcal{A_\pi}$, and $F(x,I)\subset \mathcal{A}_\pi\setminus I\subset \mathcal{A}_x$.
\end{itemize}
\end{lemma}
\begin{proof}
Suppose $\{a,b\} \in F(x,I)\setminus \mathcal{A}_\pi$. As $a,b\in Iso(I)$ and therefore the neighbours of $a$ in $\pi$ should be from set $\mathcal{N}_x(a)\setminus \{b\}$ and the neighbours of $b$ in $\pi$ should be from set $\mathcal{N}_x(b)\setminus \{a\}$, we have $|\mathcal{N}_\pi(a)|,|\mathcal{N}_\pi(b)| \leq 1$. But $|\mathcal{N}_\pi(a)|$ and $|\mathcal{N}_\pi(b)|$ cannot be $0$, since in that case $a$ or $b$ cannot be connected to the rest of the numbers to construct $\pi$, and therefore $|\mathcal{N}_\pi(a)|=|\mathcal{N}_\pi(b)|=1$ which means that $a$ and $b$ are extremities of permutation $\pi$, i.e. $\{\pi_1,\pi_n\}=\{a,b\}$. In other words, there may exist at most one adjacency $\{a,b\}\in F(x,I)\setminus \mathcal{A}_\pi$. This proves part $(a)$. For part $(b)$, suppose $\pi'\in \overline{[id,x]}$ and there exists adjacency $\{a,b\}$ such that $\{a,b\}\in F(x,I)\setminus \pi'$. As we showed above $\{\pi'_1,\pi'_n\}=\{a,b\}$. Also, as $a$ and $b$ are connected in $\pi'$ through a segment of $\pi'$ containing at least one segment of $I$ and this means that there exists at least one $<1,2>$-adjacency ($<1,2>$-segment) in the segment of $\pi'$ connecting $a$ to $b$, namely $e$, and hence $e$ is not in $F(x,I)$. Therefore, we can construct a new permutation $\pi$ by cutting $e$ in $\pi'$ and joining $a$ to $b$. This proves part $(b)$.
\end{proof}
\begin{proof}[Proof of Theorem \ref{interval-theorem}]
The left inequality holds, since, when $F(x,I)$ is an empty segment set, we need at least $\|I\|-1$ $<1,1>$-segments (trivial segments) from $x$ to complete $\pi$. To prove the right inequality, let $\pi$ be a permutation such that $I\subset \A_\pi$ and $F(x,I)\subset \A_\pi\setminus I \subset \A_x$. From Lemma \ref{lemma-J-interior}, we know that such $\pi$ exists. As the freedom of every number in any segment of $F(x,I)$ is $2$, for two segments of $F(x,I)$, say $s_1,s_2$, the segment of $\pi$ that is located between them in $\pi$, say $s$, should necessarily contain at least one segment of $I$. In fact, the freedom of $s$ cannot be $<2,2,...,2>$ (since in that case $s_1\cup s\cup s_2$ should be a segment of $F(x,I)$ that is not supposed so) and then there must be at least a number in the segment $s$ with freedom  $1$, and this implies that a segment of $I$ must be contained in $s$. This yields that two segments of $F(x,I)$ cannot be connected to each other in $\pi$ without using at least a segment of $I$ between them. On the other hand, let $s_1,s_2$ be two segments of $I$, and call the segment of $\pi$ located between them in $\pi$, $s$. If $s$ does not contain a segment of $F(x,I)$ and does not contain a segment of $I$, then it must be either a $<1,1>$-segment (i.e. a trivial segment) or a $<1,2,1>$-segment. Lastly, let $s_1$ be a segment of $I$ and $s_2$ be a segment of $F(x,I)$ and let $s$ be a segment of $\pi$ that is located between $s_1$ and $s_2$ in $\pi$. If $s$ does not contain a segment of $I$, it should be a $<1,2>$-segment necessarily. Putting all these together, we conclude that between each pair of segments of $I$ in $\pi$, say $s_1,s_2$, we may need either a $<1,2,1>$-segment of $\A_x\setminus F(x,I)$ or at most one segment of $F(x,I)$. In the latter for each end of this segment from $F(x,I)$, we need a $<1,1>$-segment of $\A_x\setminus F(x,I)$ to connect it to $s_1$ and $s_2$. On the other hand, on the right-hand side (left-hand side) of the most right (left) segment of $I$ in $\pi$, we may place either a $<1,2>$-segment ($<2,1>$-segment) or a $<1,2>$-segment ($<2,1>$-segment) followed by a segment of $F(x,I)$ on its right (on its left). So in general, we need at most $2$ adjacencies of $\A_x\setminus F(x,I)$ between each pair $s_1,s_2\hat{\in} I$ which are neighbours with respect to $\pi$ and in extremities we need at most one adjacency of $\A_x\setminus F(x,I)$. In other words, we need at most $2(\|I\|-1)+2=2\|I\|$ adjacencies of $\A_x\setminus F(x,I)$ in order to complete $\pi$. This finishes the proof.
\end{proof}
\begin{center}
\begin{figure}[!ht]
\begin{center}
\includegraphics[scale=0.6, trim = 0.3cm 13.5cm 3.5cm 2.7cm,clip]{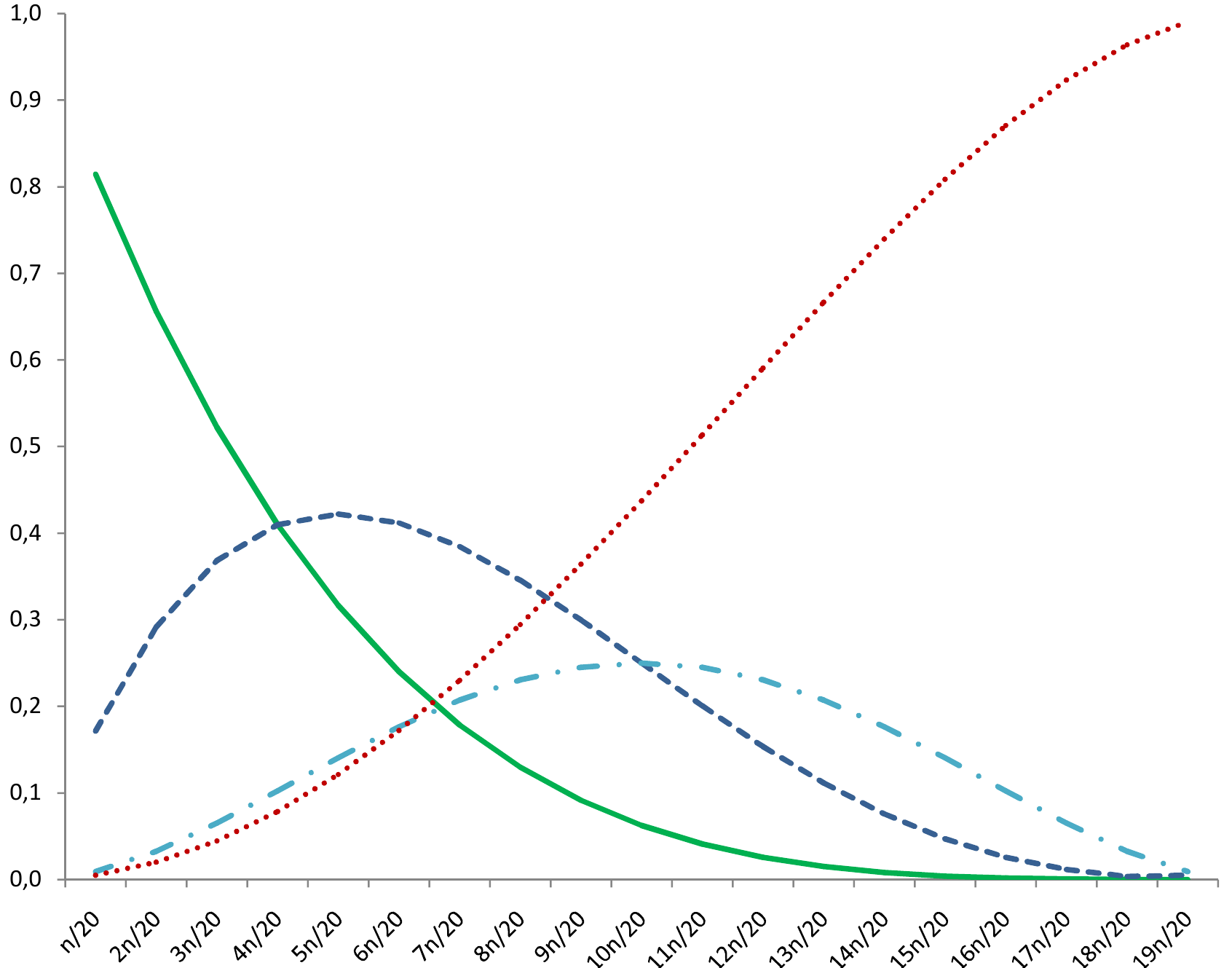}
\end{center}
\caption{The value of $\e[\alpha]/n$ (in green), $\e[\beta]/n$ (in dark blue), $\e[\delta]/n$ (in light blue) and $\e[\gamma]/n$ (in red), when we choose $\frac{n}{20}$, $\frac{2n}{20}$,$\frac{3n}{20}$...,$\frac{19n}{20}$ adjacencies of $id$.}\label{figexp}
\end{figure}
\end{center}
Let $(\hat{I}_n)_{n\in \N}$ be an arbitrary sequence of segment sets with $|\hat{I}_n|=m(n)$ and $\|\hat{I}_n\|=k(n)$ for $n\in \N$. As we already saw, to have $x\in X_n(\hat{I}_n)$, it is necessary to have $k(n)\leq m(n)\leq n-k(n)$ and also by Theorem \ref{interval-theorem}
\[
\|\hat{I}_n\|-1\leq Q(x,\hat{I}_n)\leq 2\|\hat{I}_n\|.
\]
Also, for $x\in X_n(\hat{I}_n)$, by definition we have,
\[
Q(x,\hat{I}_n)\leq \beta(x,\hat{I}_n)+\gamma(x,\hat{I}_n).
\]
Now suppose $m(n)/n\rightarrow c$ and $k(n)/n\rightarrow c'$, as $n\rightarrow \infty$, for $c,c'\in \R_+$. Then Theorem \ref{convergence in probability} implies that the right side of the above inequality converges to $4c'-4cc'$, in probability, as $n$ goes to $\infty$. Similarly, the left side of the last inequality converges to $1-c-(1-c-c')^2$, in probability, as $n\rightarrow \infty$. Now suppose $c,c'$ is such that
\[
1-c-(1-c-c')^2>4c'-4cc'.
\]
Let $\varepsilon <<1-c-(1-c-c')^2-4c'+4cc'$. Then
\[
\begin{array}{l}
\p[\xi^{(n)}\in X_n(\hat{I}_n)]\leq\\
\p[Q(\xi^{(n)},\hat{I}_n)\leq \beta(\xi^{(n)},\hat{I}_n)+\gamma(\xi^{(n)},\hat{I}_n)]\leq \\
\p[|\frac{Q(\xi^{(n)},\hat{I}_n)}{n}-(1-c-(1-c-c')^2)|>\varepsilon]+\\
\p[|\frac{\beta(\xi^{(n)},\hat{I}_n+\gamma(\xi^{(n)},\hat{I}_n)}{n}-(4c'-4cc')|>\varepsilon] \rightarrow 0,
\end{array}
\]
as $n\rightarrow 0$. So, to avoid this, we should assume $1-c-(1-c-c')^2 \leq 4c'-4cc'$. Similarly, we derive
\[
\left\{
	\begin{array}{l}
     1-c-(1-c-c')^2\leq 4c'-4cc',\\
     c'\leq 1-c-(1-c-c')^2\leq 2c',\\
     0< c'\leq c\leq 1-c'.
	\end{array}
\right.
\]

\section{Finding non-trivial partial geodesics}\label{section-main-1}

In this section we count the number of elements in $X_n(I)$ for a given segment set $I$. This gives an upper bound for the number of elements in $\bar{X}_n(I)$, by which we will be able to estimate the asymptotic behaviour of the probability of having a geodesic point of $id$ and $\xi^{(n)}$, far from both of them, as $n$ tends to $\infty$. In fact we can prove that this probability converges to $0$. This partly proves a conjecture stated by Haghighi and Sankoff in \cite{haghighi12}, for the case of two random permutations.

Recall the definition of the set of intrinsic points, end points and isolated points of a given segment set $I$ from Section~\ref{prem}, and as before denote them by $Int(I)$, $End(I)$ and $Iso(I)$, respectively. 

\begin{lemma}\label{complementary}
Let $x$ be a permutation in $S_n$, and let $I\in \mathcal{I}^{(n)}$ be a segment set
. There exist a permutation $\pi\in S_n$ containing $I$ such that $\A_{\pi}\setminus I \subset \A_x$ if and only if there exist $q,r\in [n]$ and a segment set $J$ contained in $x$ satisfying one of the following conditions:
\begin{enumerate}[(i)]
\item $\{q,r\}=End(I)\cap Iso(J)$, $\|J\|=\|I\|-1$, $Int(J)=Iso(I)$, $Iso(J)\setminus\{q,r\}=Int(I)$ and $End(J)=End(I)\setminus\{q,r\}$;
\item $\{r\}=End(J)\cap Iso(I)$ and $\{q\}=End(I)\cap Iso(J)$, $\|J\|=\|I\|$, $Int(J)=Iso(I)\setminus\{r\}$, $Iso(J)\setminus\{q\}=Int(I)$ and $End(J)\setminus\{r\}=End(I)\setminus\{q\}$; or
\item $\{q,r\}=End(J)\cap Iso(I)$, $\|J\|=\|I\|+1$, $Int(J)=Iso(I)\setminus\{q,r\}$, $Iso(J)=Int(I)$ and $End(J)\setminus\{q,r\}=End(I)$.
\end{enumerate}

\noindent In any of these three cases, $q$ will be $\pi_1$ and $r$ will be $\pi_n$, or the opposite.
\end{lemma}

\begin{proof}
To prove necessity, let $\pi$ be a permutation in $S_n$ containing $I$ such that $\A_{\pi}\setminus I \subset \A_x$, and define $J:=\A_{\pi}\setminus I$. Then $I$ and $J$ are disjoint and they complete each other in an alternating way, that is for any pair of neighbour segments $s_1,s_2\hat{\in}I$ with respect to $\pi$, there exists exactly one segment of $J$ that connects $s_1$ and $s_2$, and similarly, for any pair of neighbour segments $s_1',s_2'\hat{\in} J$ with respect to $\pi$, there exists exactly one segment of $I$ that connects $s_1'$ and $s_2'$. Therefore, we have $|\|I\|-\|J\||\leq 1$, and also, all intrinsic points of $I$ must be isolated points of $J$, $Int(I)\subseteq Iso(J)$, as well as all intrinsic points of $J$ must be isolated points of $I$, $Int(J)\subseteq Iso(I)$. Furthermore, all end points of $I$, except at most two of them, must be end points of $J$, and similarly, all end points of $J$, except at most two of them must be end points of $I$. Indeed, when we remove the intersection of end points of $I$  and end points of $J$ from the end points of the union, two points remain. In other words, there exists two points $q,r\in [n]$ such that
\[
End(I \cup J)\setminus (End(I)\cap End(J))=\{q,r\}.
\]
These two points can either be both end points of $I$, or both end points of $J$, or one of them an end point of $I$ and the other an end point of $J$ according to the following cases. 
\begin{itemize}
\item[(i)] If $\|J\|=\|I\|-1$, then $\{\pi_1,\pi_2\}$ and $\{\pi_{n-1},\pi_n\}$ are adjacencies of $id$, and so $q:=\pi_1$ and $r:=\pi_n$ are end points of $I$ while both are isolated points of $J$. Therefore, we have $Int(J)=Iso(I)$, $Iso(J)\setminus\{q,r\}=Int(I)$ and $End(J)=End(I)\setminus\{q,r\}$. 

\item[(ii)] If $\|I\|=\|J\|$, then either $\{\pi_1,\pi_2\}$ is an adjacency of $id$ and $\{\pi_{n-1},\pi_n\}$ is an adjacency of $x$ or vice versa, $\{\pi_1,\pi_2\}$ is an adjacency of $x$ and $\{\pi_{n-1},\pi_n\}$ is an adjacency of $id$. Without loss of generality suppose $\{\pi_1,\pi_2\}$ is an adjacency of $id$ and $\{\pi_{n-1},\pi_n\}$ is an adjacency of $x$. Then $q:=\pi_1$ is an end point of $I$ and also an isolated point of $J$, while $r:=\pi_n$ is an end point of $J$ and also an isolated point of $id$ with respect to $I$, and we have $Int(J)=Iso(I)\setminus\{r\}$, $Iso(J)\setminus\{q\}=Int(I)$ and $End(J)\setminus\{r\}=End(I)\setminus\{q\}$. 

\item[(iii)] Finally, if $\|J\|=\|I\|+1$, then $\{\pi_1,\pi_2\}$ and $\{\pi_{n-1},\pi_n\}$ are adjacencies of $x$. Therefore, $q:=\pi_1$ and $r:=\pi_n$ are end points of $J$ and also isolated points of $I$. Furthermore, $Int(J)=Iso(I)\setminus\{q,r\}$, $Iso(J)=Int(I)$ and $End(J)\setminus\{q,r\}=End(I)$.
\end{itemize}
To prove sufficiency, let $q,r\in [n]$ and $J$ be a segment set contained in $x$ satisfying condition $(i)$ in the statement of the lemma (the proof is similar, for $q,r,$ and $J$ satisfying conditions $(ii)$ and $(iii)$). Then
\[
Int(I)\cup End(I) \cup Int(J)=Int(I)\cup End(I) \cup Iso(I)=[n],
\]
and
\[
Int(I)\cap Int(J)=Int(I)\cap Iso(J)=\emptyset .
\]
In fact, this shows that $I$ and $J$ complete each other in an alternating way, and $I\cup J$ is a unique segment with extremities $q$ and $r$, i.e. $End(I\cup J)=\{q,r\}$, and with intrinsic points $Int(I\cup J)=[n]\setminus \{q,r\}$. In other words, there exists a permutation $\pi$ such that $\mathcal{A}_\pi=I\cup J$. As $I$ and $J$ are disjoint, one can write $J=\mathcal{A}_\pi\setminus I\subset \mathcal{A}_x$. This finishes the proof.

\end{proof}

Let $I$ be a segment set in $\mathcal{I}^{(n)}$, and let $x\in X_n(I)$. From Lemma \ref{complementary}, $x$ contains a segment set $J$ satisfying one of the three conditions indicated in the statement of Lemma~\ref{complementary}.



\begin{remark}
Let $I$ be a segment set of $id$ and $\pi$ a permutation containing $I$. In order to construct a permutation $x$ such that $\bar{I}_\pi=\mathcal{A}_\pi\setminus I \subset \mathcal{A}_x$, we should take different rearrangements of segments of $\overline{I}_{\pi}$ (considering two directions) and intrinsic points of $I$. Each such rearrangement gives us a permutation $x\in X_n(I)$. 
\end{remark}

In Theorem~\ref{thm:permnumb}, we give an explicit formula for the number of permutations in $X_n(I)$ as a function of the number of adjacencies and segments in $I$. To this end, we need the following lemma.




\begin{lemma}\label{permutation}
Given a segment set $I$ with $m$ adjacencies and $k$ segments, that is $I\in \mathcal{I}_{m,k}^{(n)}$, the number of permutations in $S_n$ containing $I$ is equal to  $2^{k}(n-m)!$.
\end{lemma}

\begin{proof}
As the segment set $I$ has $m$ adjacencies and $k$ segments, each permutation containing $I$ has $n-m-k$ isolated points with respect to $I$. Therefore, noting that segments have two directions, we have $2^{k}(k+(n-m-k))!$ permutations containing $I$.
\end{proof}


\begin{theorem}\label{thm:permnumb}
Given a segment set $I$ with $m$ adjacencies and $k$ segments, that is $I\in \mathcal{I}_{m,k}^{(n)}$, we have:
\begin{multline}
|X_n(I)|= \frac{2^k(m+1)!(n-m-2)!}{k!}\\
\times\left(k^2(k-1)+2k(n-m-k)+\frac{(n-m-k)(n-m-k-1)}{k+1}\right)
\end{multline}

\end{theorem}

\begin{proof}
Note that since the segment set $I$ has $m$ adjacencies and $\|I\|=k$, then $|Int(I)|=m-k$, $|Iso(I)|=n-m-k$ and $|End(I)|=2k$. By definition, $x\in X_n(I)$ if there exist a segment set $J$ that satisfies one of the three conditions in Lemma~\ref{complementary}. We divide the proof into three cases. We shall count the number of ways we can construct $J$ for each one of the three cases
, and thus, we use Lemma~\ref{permutation} to compute the number of permutations $x\in X_n(I)$ containing $J$ in each case. 

If $\|J\|=k-1$, then to have a permutation $\pi$ such that $\mathcal{A}_\pi$ is a sequence of alternating segments from $I$ and $J$, the number of ways that we can choose pairs of end points to construct segments of $J$ is equal to the number of ways we can rearrange the segments of $I$, noting that each segment can be placed in two different directions and $End(J)\subset End(I)$. Hence, we have $2^{k}k!$ ways to choose pairs of end points for $J$. On the other hand, when the end points of segments of $J$ are fixed, as $Int(J)=Iso(I)$, the number of ways that one can distribute (with order) $n-m-k$ intrinsic points in $k-1$ segments of $J$  is $\frac{((n-m-k)+(k-2))!}{(k-2)!}=\frac{(n-m-2)!}{(k-2)!}$. Ignoring the direction and order of segments in this calculation, we have 
\[
2^kk!\frac{(n-m-2)!}{(k-2)!(k-1)!2^{k-1}}=2k\frac{(n-m-2)!}{(k-2)!}
\] 
ways to construct segment set $J$. Remember that each of these possible segment sets $J$ has exactly $k-1$ segments and $n-m-1$ adjacencies and therefore, applying Lemma~\ref{permutation}, there exist 
\[
2k\frac{(n-m-2)!}{(k-2)!}2^{k-1}(n-(n-m-1))!=\frac{2^k k (n-m-2)!(m+1)!}{(k-2)!}
\]
permutations $x\in X_n(I)$ containing $J$ that satisfies the case $(i)$ of Lemma~\ref{complementary}.

Similarly, if $\|J\|=k$, the number of ways that we can choose pairs of end points for segments of $J$ is equal to the number of ways we can arrange the segments of $I$, noting that each segment can be in two directions and in this case one of the end points of $J$ must be chosen from $Iso(I)$ since $End(J)\setminus\{r\}=End(I)\setminus\{q\}$ where $r$ is an end point of $J$ and an isolated point in $id$ with respect to $I$, and $q$ is an end point of $I$ and an isolated point in $x$ with respect to $J$. Therefore, we have $2(n-m-k) 2^{k}k!$ ways to choose pairs of end points in order to construct $J$. Whereas, $|Int(J)|=|Iso(I)|-1$, we have 
\[
2(n-m-k)2^{k}k!\frac{((n-m-k-1)+k-1)!}{(k-1)!2^kk!}=\frac{2(n-m-k)(n-m-2)!}{(k-1)!}
\]
ways to construct segment set $J$. Thus there exist 
\[
2^{k}(m+1)!\frac{2(n-m-k)(n-m-2)!}{(k-1)!}=\frac{2^{k+1}(n-m-k)(m+1)!(n-m-2)!}{(k-1)!}
\]
permutations $x\in X_n(I)$ containing $J$ that satisfies case $(ii)$ of Lemma~\ref{complementary}.

Lastly, if $\|J\|=k+1$ then $|Int(J)|=n-m-k-2$, $|Iso(J)|=m-k$ and $End(J)\setminus\{q,r\}=End(I)$ where $q$ and $r$ are end points of $J$ and isolated points of $I$. Therefore, similarly, there exist 

\begin{multline}
2^kk!(n-m-k)(n-m-k-1)\frac{(n-m-2)!}{k!(k+1)!2^{k+1}}2^{k+1}(m+1)!=\\
\frac{2^k(n-m-k)(n-m-k-1)(n-m-2)!(m+1)!}{(k+1)!}
\end{multline}

permutations $x\in X_n(I)$ containing $J$ that satisfies case $(iii)$ of Lemma~\ref{complementary}.
\end{proof}
\begin{remark}[Random segment set]
Applying Proposition~\ref{segmentset}, the probability of existence of a permutation $\pi\in \overline{[id,\xi^{(n)}]}$ containing random segment set $I_m^{(n)}$ such that $\A_{\pi}\setminus I_m^{(n)} \subset \A_{\xi^{(n)}}$, is bounded by 
\begin{multline*}
\p(\xi^{(n)}\in X_n(I_m^{(n)}))=\sum\limits_{k=1}^m \frac{{m-1 \choose k-1}{n-m \choose k} 2^k(m+1)!(n-m-2)!}{{n-1 \choose m} k!n!}\\
\times\left(k^2(k-1)+2k(n-m-k)+\frac{(n-m-k)(n-m-k-1)}{k+1}\right).
\end{multline*}
\qed
\end{remark}

For $0< \varepsilon<1/2$, let 
\[
\Lambda_n^\varepsilon:=\bigcup\limits_{m\leq n-1} \bigcup\limits_{k\geq \varepsilon n} \mathcal{I}_{m,k}^{(n)}.
\]
Note that the condition $k\geq l$, for convenient $l\in [n]$, implies that $l\leq m\leq n-l$, since $k\leq m$ and also in order that a segment set $I$ contained in a permutation $x$ has at least $l$ segments, at least $l-1$ adjacencies of $x$ should not appear in $I$. The following theorem is the consequence of Theorem \ref{permutation}.
\begin{theorem}\label{thm:main}
Let $0<\varepsilon<1/2$ and let $(I_n)_{n\in \N}$ be a sequence of segment sets such that $I_n\in \mathcal{I}^{(n)}$ and $\varepsilon n\leq |I_n| \leq (1-\varepsilon) n$. Then
\[
\frac{|X_n(I_n)|}{n!} \rightarrow 0,
\]
as $n\rightarrow \infty$. Furthermore,
\[
\p(\xi^{(n)}\in \bigcup\limits_{I\in \Lambda_n^\varepsilon} X_n(I))\rightarrow 0,
\]
as $n\rightarrow \infty$.
\end{theorem}
\begin{proof}
By assumption, for every $n\in \N$, there exists $c_n$ such that $\varepsilon\leq c_n  \leq 1-\varepsilon$ and $|I_n|=nc_n+o(n)$. Then, by Lemma \ref{permutation} and Stirling's formula, there exists a constant $c_0$ such that
\[
\begin{array}{l}
\lim\limits_{n\rightarrow \infty} \frac{|X_n(I_n)|}{n!}\\
\leq c_0 \lim\limits_{n\rightarrow \infty} \frac{ (\frac{c_n n}{e})^{c_n n+o(n)}(\frac{(1-c_n)n}{e})^{(1-c_n)n-o(n)}}{(\frac{n}{e})^n}(n^{\frac{7}{2}}+o(n^{\frac{7}{2}}))\\
\leq c_0 \lim\limits_{n\rightarrow \infty} (\varepsilon^\varepsilon (1-\varepsilon)^{1-\varepsilon})^{n+o(n)} (n^{\frac{7}{2}}+o(n^{\frac{7}{2}}))=0,
\end{array}
\]
where the last inequality holds as the maximum of the function $f(x)=x^x(1-x)^{1-x}$ in the domain $[\varepsilon,1-\varepsilon]$ is $\varepsilon^\varepsilon (1-\varepsilon)^{1-\varepsilon}$.

For the second part, recall that if $I\in \Lambda_n^\varepsilon$, then $\|I\|\geq \varepsilon n$, and hence, $\varepsilon n\leq |I|\leq (1-\varepsilon)n$. For any $I\in \Lambda_n^\varepsilon$, from Theorem \ref{thm:permnumb}, we have
\[
\frac{|X_n(I)|}{n!} \leq \frac{2^{\lfloor (1-\varepsilon)n\rfloor +1}\lfloor \varepsilon n\rfloor ! (\lfloor (1-\varepsilon)n\rfloor+1)!}{\lfloor \varepsilon n\rfloor ! n!}(n^3+o(n^3)).
\]

 Therefore, $|\Lambda_n^\varepsilon|\leq 2^{n-1}$ and Stirling's formula imply 
\[
\begin{array}{l}
\lim\limits_{n\rightarrow \infty} \p(\xi^{(n)}\in \bigcup\limits_{I\in \Lambda_n^\varepsilon} X_n(I))\\
\leq \lim\limits_{n\rightarrow \infty} \frac{2^n (2e)^{(1-\varepsilon)n+o(n)} (\varepsilon^\varepsilon (1-\varepsilon)^{1-\varepsilon})^n}{(\varepsilon n)^{\varepsilon n}}(n^3+o(n^3))=0.\\

\end{array}
\]
\end{proof}
Now we prove the main theorem of this section, namely, we prove, in part, a conjecture stated in Haghighi et. al. \cite{haghighi12}. For $\varepsilon>0$, set
\[
\mathcal{D}_n^\varepsilon:=\{x\in S_n : \exists \pi \in \overline{[id,x]} \ s.t. \ \ d^{(n)}(\pi,id),d^{(n)}(\pi,x)\geq \varepsilon n\}.
\]
Also, for $a\in \R$, define
\[
\Delta_n^a:=\{x\in S_n : |\mathcal{A}_{id,x}|\leq a\}.
\]
\begin{theorem}\label{thm:main-main}
For any $\varepsilon >0$,
\[
\p(\xi^{(n)}\in \mathcal{D}_n^\varepsilon)\rightarrow 0,
\]
as $n\rightarrow 0$.
\end{theorem}
\begin{proof}
Let $(a_n)_{n\in \N}$ be an arbitrary sequence of real numbers diverging to $\infty$ such that $a_n / n\rightarrow 0$, as $n\rightarrow \infty$. Let 
\[
\Upsilon_n^\varepsilon:=\bigcup\limits_{m\in [\frac{\varepsilon}{2}n,(1-\frac{\varepsilon}{2})n]} \mathcal{I}_m^{(n)}
\]
Then
\[
\begin{array}{l}
0\leq \lim\limits_{n \rightarrow\infty}\p(\xi^{(n)}\in \mathcal{D}_n^\varepsilon)=\lim\limits_{n \rightarrow\infty}\p(\xi^{(n)}\in \mathcal{D}_n^\varepsilon \cap \Delta_n^{a_n})\leq \\
\lim\limits_{n \rightarrow\infty}\p(\xi^{(n)}\in \Delta_n^{a_n} \cap \bigcup\limits_{I\in \Upsilon_n^\varepsilon} X_n(I))\leq\\
\lim\limits_{n\rightarrow\infty}\p(\xi^{(n)}\in \bigcup\limits_{I\in \Upsilon_n^\varepsilon} X_n(I))=0,\\
\end{array}
\]
where the last convergence holds from Theorem \ref{thm:permnumb}, Theorem \ref{thm:main}, and Stirling's formula.
\end{proof}

\end{document}